\documentclass[12pt]{amsart}
\usepackage{amssymb,amsmath}
\usepackage{geometry}    
\usepackage{mathrsfs}

\usepackage{amsfonts}
\usepackage{bbm}

\usepackage{graphicx}

\usepackage{vmargin}
\setmargrb{1in}{1in}{1in}{1in}

\newtheorem{theorem}{Theorem}[section]
\newtheorem{lemma}[theorem]{Lemma}
\newtheorem{proposition}[theorem]{Proposition}
\newtheorem{corollary}{Corollary}
\theoremstyle{definition}
\newtheorem{definition}{Definition}
\newtheorem{remark}{Remark}

\newtheorem{conjecture}{Conjecture}
\newtheorem{problem}{Problem}

\begin{document}

\title[Metric results on sumsets and Cartesian products]{Metric results on sumsets and Cartesian products of classes
 of Diophantine sets}

\author{Johannes Schleischitz}

\thanks{Middle East Technical University, Northern Cyprus Campus, Kalkanli, G\"uzelyurt \\
    johannes@metu.edu.tr ; jschleischitz@outlook.com}

\begin{abstract}
	Erd\H{o}s proved that any real number can be written
	as a sum, and a product, of two Liouville numbers. 
	Motivated by these results, we study sumsets of classes
	of real numbers with prescribed (or bounded)
	irrationality exponents. We show that such sumsets
	turn out to be large in general, indeed
	almost every real number with respect to Lebesgue measure 
	can be written as the
	sum of two numbers with sufficiently large prescribed irrationality
	exponents. In fact the Hausdorff dimension of the complement
	is small, and the result remains true if we impose considerably refined conditions on the orders of rational approximation
	(``exact approximation'' with respect to an approximation function).
	As an application, we show that in many cases 
	the Hausdorff dimension of
    Cartesian products of sets with prescribed irrationality exponent 
    exceeds the expected dimension,
    that is the sum of the single Hausdorff dimensions. 
    We also address their packing dimensions.
    Similar results hold when restricting to classical
    missing digit Cantor sets, relative to its natural Cantor measure.
    In particular, we prove that the subset
    of numbers with prescribed large irrationality exponent 
    has full packing dimension, i.e. the same packing dimension as the 
    entire Cantor set.
    This complements some results on the Hausdorff
    dimension of these sets, which is an extensively studied topic
    in Diophantine approximation.   
    Our proofs are based on ideas of Erd\H{o}s, but vastly extend them.
\end{abstract}

\maketitle

{\footnotesize{

{\em Keywords}: Hausdorff dimension, packing dimension, irrationality exponent, continued fractions\\
Math Subject Classification 2010: 11J04, 11J82, 11J83}}


\section{Hausdorff dimensions and Cartesian products}  \label{int01}

Since we aim to establish metrical results, we start by recalling 
some fundamental metric concepts.
Hausdorff measure and Hausdorff
dimension are widely used concepts in measuring the size
of a set. 
We will throughout denote by $\dim(A)$ the
Hausdorff dimension of the set $A$. We will sporadically deal with
the packing dimension $\dim_{P}(A)$ of a $A\subseteq \mathbb{R}^{n}$ as well. We omit
the exact definitions and refer to Falconer~\cite{falconer}.
We remind that the Hausdorff dimension of a set
never exceeds its packing dimension.

We investigate Hausdorff dimensions of sumsets and 
Cartesian products of certain Euclidean sets. Let us focus 
on Cartesian products for now.
This topic has been addressed for various classes of sets, see for example,
Besicovitch and Moran~\cite{moran}, Eggleston~\cite{egg}, 
Marstrand~\cite{mars}, Xiao~\cite{xiao}, Larman~\cite{larman}, Wegmann~\cite{wegmann}.   
A fundamental property of Hausdorff dimension proved
by Marstrand~\cite{mars} is that for any 
measurable sets $A,B$ 
when taking their Cartesian product we have
\begin{equation} \label{eq:most}
\dim(A\times B) \geq \dim A+\dim B.
\end{equation}
In general, we do not have equality in \eqref{eq:most}.
However, in many interesting situations, the equality does hold, for example,
for products of classical fractals like the 
Cantor middle-third set~\cite{falconer}.
Critria on the sets $A,B$ that imply that the equality 
holds can be found 
as well in~\cite{falconer}.
An upper bound due to Tricot~\cite{tricot}
for the left hand side of \eqref{eq:most}
involving the packing dimension
is contained in Theorem~\ref{uppack}
below.

\begin{theorem}[Tricot] \label{uppack}
	For any measurable sets $A,B$ in $\mathbb{R}^{n}$, we have 
	\[
	\dim(A\times B)\leq \dim(A)+\dim_{P}(B).
	\]
	Hence, if $A_{1},\ldots,A_{n}$ are measurable 
	subsets of $\mathbb{R}$, then
	\[
	\dim(A_{1}\times A_{2}\times \cdots \times A_{n}) \leq n-1+ \min_{1\leq i\leq n} \dim(A_{i}). 
	\]  
\end{theorem}

See also Bishop, Peres~\cite{bishop} for refinements.
For completeness we also want to mention the estimates
\begin{equation}  \label{eq:uppak}
    \dim_P(A)+ \dim_P(B)\ge \dim_P(A\times B) \ge \dim(A) + \dim_P(B)
\end{equation}
due to Tricot~\cite{tricot} as well, which we will not use as frequently in this paper.
One purpose of this paper is to find sets that naturally occur in Diophantine
approximation, whose Cartesian products are of Hausdorff dimension
strictly larger than the sum of the single dimensions, i.e. there is no
equality in \eqref{eq:most}. On the way, we will emcompass 
sumsets and study their properties.
An important tool to achieve this goal is 
the following rather elementary property of
Hausdorff measures and dimensions, applied in suitable contexts.

\begin{proposition} \label{pp}
	Let $A\subseteq \mathbb{R}^{m}$ be a measurable set
	and $\phi: \mathbb{R}^{m}\to \mathbb{R}^{n}$
	be Lipschitz. Then $\dim \phi(A)\leq \dim(A)$. More generally,
	for any $s\geq 0$ writing $H_{s}$ for the $s$-dimensional Hausdorff
	measure, we have $H_{s}(\phi(A)) \ll_{s,m,n} H_{s}(A)$.
\end{proposition}

See~\cite[Proposition~2.2, Corollary~2.4]{falconer} and also
\cite[Proposition 2.2]{hss} for a more general version. In 
Proposition~\ref{pp},
and the sequel, we use Vinogradov's notation $A\ll_{.} B$ which means 
$A\leq c(.)B$, that is $A$ does not exceed $B$ by more than some multiplicative constant depending on the subscript variables
only, with an absolute constant if no subscript occurs. As usual, we shall also use $A\asymp B$ as short notation for $A\ll B\ll A$.

\section{Sumsets and Cartesian products of the set of Liouville numbers} \label{1}

Even though the deepest results of the paper appear in 
\S~\ref{ch3}, we prefer to start our investigation with
Cartesian products of Liouville numbers, where the historcial context
and motivation can be presented more naturally.

Recall that $\xi\in\mathbb{R}\setminus \mathbb{Q}$ is called 
Liouville number if the inequality
\[
|\xi-\frac{p}{q}| \leq q^{-N}
\]
has solutions in rational numbers $p/q$ for arbitrarily large $N$.
Let us denote the set of Liouville numbers by $\mathscr{L}$. 
The Hausdorff dimension of $\mathscr{L}$ equals to $0$. But,
on the other hand, it is co-meager, i.e. its complement $\mathbb{R}\setminus \mathscr{L}$ 
is of first category. See Chapter 2 of Oxtoby's book~\cite{oxtoby}  
for short proofs of both results.
For refined further measure theoretic results on
$\mathscr{L}$ when considering general Hausdorff $f$-measures, we 
refer to Olsen and Renfro~\cite{olsen} and Bugeaud, Dodson and Kristensen~\cite{bdk}.

A well-known result of Erd\H{o}s~\cite{erdos}
that motivates the investigations in this paper
claims that every
real number can be written as a sum (or product if $\ne 0$)
of two Liouville numbers.
Erd\H{o}s gave two proofs. One is based on the mentioned
fact that $\mathscr{L}$ 
is co-meager. Indeed, consequently the set
$\mathscr{L}\cap \mathscr{L}_{\xi}$ with 
$\mathscr{L}_{\xi}=\{\xi-x: x\in\mathscr{L}\}$ is co-meager
as well for any $\xi\in\mathbb{R}$, thus non-empty. Now any pair $(y,\xi-y)$   
with $y$ in the intersection $\mathscr{L}\cap \mathscr{L}_{\xi}$ consists of Liouville numbers 
that by construction sum up to a given $\xi$.
The argument
can be widely extended,
see Rieger~\cite{rieger}, Schwarz~\cite{schwarz}, Burger~\cite{b1, b2}
and Senthil Kumar, Thangadurai, Waldschmidt~\cite{kumar}.
For the second proof, Erd\H{o}s effectively constructs 
Liouville numbers $x,y$ with the property that $x+y=\xi$
for any given $\xi\in\mathbb{R}$. Let us recall this proof as well. Suppose $\xi$
has decimal expansion $\xi=0.d_{1}d_{2}\ldots$, $0\le d_j\le 9$. 
Define $b_{j}=j!$. Let $x$ be the number with the same
base 10 digits of $\xi$ for indices from $b_{2j}+1$ to $b_{2j+1}$, and $0$
otherwise, and let $y$ be the number having the digits of $\xi$ in the remaining intervals from $b_{2j+1}+1$ to $b_{2j}$ and $0$ otherwise. 
Then $x+y=\xi$. On the other hand,
$x,y$ are both Liouville numbers. Indeed, the rational numbers
obtained from cutting off the decimal expansion 
of $x$ and $y$ after positions
of the form $b_{2j+1}$
and $b_{2j}$ respectively, will be very good rational approximations to $x$
and $y$, respectively. (One thing unnoticed by Erd\H{o}s is that
$x$ or $y$ could potentially be rational. However,
the method is flexible enough to overcome this problem by a short variation argument.)

Now observe the following consequence of Erd\H{o}s' result
when combined with Proposition~\ref{pp}: Since the map
\begin{align*}
\mathscr{L}\times \mathscr{L} &\longmapsto \mathbb{R} \\
(x,y)&\longmapsto x+y
\end{align*}
is Lipschitz continuous and surjective, the product set $\mathscr{L}\times \mathscr{L}$ has Hausdorff dimension at least $1$, even though $\mathscr{L}$
has Hausdorff dimension $0$. 
In fact
\begin{equation} \label{eq:spezi}
\dim(\mathscr{L}\times \mathscr{L})=1,
\end{equation}
since the reverse bound follows from Theorem~\ref{uppack}.
Though \eqref{eq:spezi} 
is an easy implication of Erd\H{o}s' result, the author
did not find this fact explicitly in the literature.
We should remark that \eqref{eq:spezi} applies, by the same argument, to any subset $L$ of $\mathscr{L}$ with the property that $L+L=\mathbb{R}$.
As pointed out to me by Sidney A. Morris, it has recently been proved that there is an abundance
of these sets, see~\cite{morris} for details.
In the sequel, we write $A^{n}$
for the $n$-fold Cartesian product $A\times A \cdots \times A$
of a set $A$.
We use a similar idea to generalise \eqref{eq:spezi}
to the $n$-fold product $\mathscr{L}^n$. 
The packing dimension of $\mathscr{L}^n$ will also be calculated.

\begin{theorem} \label{th01}
	For any integer $n\geq 1$, the set $\mathscr{L}^{n}$ has Hausdorff dimension $n-1$ and packing dimension $n$.
\end{theorem}

The proof of the theorem is not difficult.
As Erd\H{o}s, we can provide two different proofs.
The first shorter one is essentially a special case of~\cite{rieger} or~\cite{schwarz}, 
the latter constructive proof prepares
the reader for the more complicated constructions in the proofs of our results stated in \S~\ref{ch3}.

\begin{proof}
	We need only to show the identity for Hausdorff dimension, the claim 
	on packing dimension then follows from Theorem~\ref{uppack} 
	and the fact that $\dim(\mathscr{L})=0$ via
	\begin{equation} \label{eq:nnplus1}
	n=\dim(\mathscr{L}^{n+1})\leq \dim(\mathscr{L})+\dim_{P}(\mathscr{L}^{n})= \dim_{P}(\mathscr{L}^{n})\le n.
	\end{equation}
	The above inequalities 
	clearly also imply $\dim(\mathscr{L}^n)\leq n-1$
	for $n\ge 1$. 
    For the lower bound we give two proofs again, each showing 
    in a different way that the Lipschitz map
    \begin{align*}
    \Psi: \mathscr{L}^{n} &\longmapsto \mathbb{R}^{n-1}   \\
    (x_{0},x_{1},\ldots,x_{n-1}) &\longmapsto (x_{0}+x_{1}, x_{0}+x_{2},\ldots,x_{0}+x_{n-1}),
    \end{align*}
    is surjective. Then by Proposition~\ref{pp} we obtain $\dim(\mathscr{L}^n)\geq n-1$, as desired.
    
    First we see that for any real vector $\boldsymbol{\xi}=(\xi_{1},\ldots,\xi_{n-1})$ the 
    intersection
    \[
    \mathscr{F}:=\bigcap_{i=1}^{n-1} \mathscr{L}_{\xi_{i}} \cap \mathscr{L}, \qquad \text{where}\; 
    \mathscr{L}_{\xi}:= \xi-\mathscr{L}=\{ \xi-\ell: \ell\in\mathscr{L} \},
    \]
    is co-meager since every set in the intersection 
    is co-meagre. 
    In particular $\mathscr{F}$ is non-empty.
    Now it is again easy to check that any element $\boldsymbol{\xi}\in \mathscr{F}$
    induces a vector $\boldsymbol{\ell}\in\mathscr{L}^{n}$ with $\Psi(\boldsymbol{\ell})=\boldsymbol{\xi}$.
    
     We sketch a second, constructive proof.   
Let $\xi_{1},\ldots,\xi_{n-1}\in [0,1)$ be arbitrary with decimal expansions $\xi_{i}=0.d_{i,1}d_{i,2}\ldots$, $0\le d_{i,j}\le 9$.
We use a similar argument to that of
Erd\H{o}s. Let $b_{j}=j!$ for $j\geq 1$
and partition $\mathbb{N}$
into intervals of the form $I_0=\{ 1\}$ and 
$I_{j}=\{b_{j}+1,b_{j}+1,\ldots,b_{j+1}\}$ for $j\ge 1$.
Define $x_{0}$ as follows. For every $i\in\{1,2,\ldots,n-1\}$,
if $j\equiv i \bmod n$ then 
take the decimal digits in places $\ell\in I_{j}$ of $x_{0}$ to be those 
$d_{i,\ell}$ of $\xi_{i}$ in this interval.
For $j\equiv 0 \bmod n$, we define the decimal digits in places
$\ell\in I_{j}$ as $0$. Then $x_{0}$ is well-defined.
Now, we claim $x_0$ as well as $x_{i}:=\xi_{i}-x_{0}$  
for $1\leq i\leq n-1$ are
Liouville numbers. 
Once we have proved this claim, then by the arbitrariness of $\xi_1,\ldots,\xi_{n-1}$, we conclude that the map $\Psi$
is surjective. Now let us prove the claim.
For $x_0$ and any $j\equiv 0\bmod n$, define the rational number $p_j/q_j$ 
by cutting off the digits of $x_0$ after $I_{j-1}$. 
Then $q_j=10^{b_j}$ and $|x_{0}-p_{j}/q_{j}|\leq 10^{-b_{j+1}}$ since the decimal digits of $x_0$ in $I_j$ are $0$. 
As $b_{j+1}/b_{j}$ tends to infinity this clearly implies $x_0\in\mathscr{L}$.
For $x_{i}$ with $i>0$, we similarly cut off after the
last decimal digit in the respective intervals $I_{j-1}$ with $j\equiv i\bmod n$ to obtain some rational $k_{j}/10^{b_j}$, and do the same
for $\xi_i$ to obtain some rational $t_{j}/10^{b_j}$. Then
let $p_{j}/q_j=(t_{j}-k_{j})/10^{b_j}$.
Since by construction $x_0$ and $\xi_i$ have the same
digits within $I_j$, we may write 
$x_{i}=\xi_{i}-x_{0}=p_{j}/q_j+v_{j,1}-v_{j,2}$ where
the decimal expansions of the real numbers 
$v_{j,1}$ resp. $v_{j,2}$ are obtained from 
taking digit $0$ up to last
position in $I_j$ and reading the digital expansion of $\xi_i$ resp. $x_0$ onwards. 
Hence
we again readily verify
\[
|x_{i}-p_{j}/q_{j}|=|v_{j,1}-v_{j,2}|\leq |v_{j,1}|+|v_{j,2}|
\leq 2\cdot 10^{-b_{j+1}},
\qquad 1\le i\le n-1,\; j\equiv i\bmod n,
\]
and we conclude as for $i=0$ above that $x_i\in\mathscr{L}$
(again, we can easily
exclude that some $x_{i}\in\mathbb{Q}$ by a minor variation in our choice
of the $b_{j}$).   
\end{proof}

As noticed above, the main step on the surjectivity of $\Psi$ can be considered
a special case of Rieger~\cite{rieger} or
Schwarz~\cite{schwarz}. Both show that for
any continuous open maps $f_{1},\ldots,f_{r}$ on $(0,1)$,
there is $\xi\in\mathscr{L}$ with all $f_{j}(\xi)$ again in $\mathscr{L}$
(according to~\cite{schwarz} we may even take countably many $f_{j}$).
Taking $r=n-1$ and $f_{j}(x)=\xi_{j}-x$ for $1\leq j\leq n-1$ yields the surjectivity of $\Psi$. The proof in~\cite{schwarz}
uses the same method as our first proof.

We want to briefly discuss two variants of Theorem~\ref{th01}.
First,
let $C_{b,W}$ denote the classical missing digit
Cantor sets consisting of numbers that can be written
\[
\sum_{i=1}^{\infty} a_{i}b^{-i}, \qquad a_{i}\in W,
\]
where $b\geq 3$ is an integer 
and $W\subseteq \{0,1,\ldots,b-1\}$. Then, for $n=2$,
upon minor modifications our argument implies
\begin{equation}  \label{eq:minorett}
\dim((\mathscr{L}\cap C_{b,W})\times (\mathscr{L}\cap C_{b,W}))=  \dim(C_{b,W})= \frac{\log |W|}{\log b},
\end{equation}
and 
\[
\dim_{P}((\mathscr{L}\cap C_{b,W})\times (\mathscr{L}\cap C_{b,W}))=2\dim_{P}(C_{b,W}) =\frac{2\log |W|}{\log b}.
\] 
See Theorem~\ref{PREIS} below for a generalisation and further comments.
However, for $n\geq 3$ the corresponding claims are unclear.
\begin{problem}
	For $n\geq 3$, do we have
$\dim((\mathscr{L}\cap C_{b,W})^n) = (n-1)\log |W|/\log b$ and
$\dim_{P}((\mathscr{L}\cap C_{b,W})^n) = n\log |W|/\log b$?
\end{problem}

Secondly, 
for any $m\geq 1$, a very similar idea applies to
the set $\mathscr{L}_{m}$
of $m$-dimensional Liouville vectors, defined similarly
as the classical Liouville numbers $\mathscr{L}_1= \mathscr{L}$
by imposing that $|\boldsymbol{p}/q-\boldsymbol{\xi}|<q^{-N}$ has
infinitely many solutions in rational vectors $\boldsymbol{p}/q=(p_1/q,\ldots,p_m/q)\in\mathbb{Q}^{m}$, for all $N$. 
The same 
digit construction as in the proof of Theorem~\ref{th01}
simultaneously applied to all components $\xi_{1},\ldots,\xi_{m}$
of $\boldsymbol{\xi}$ (i.e. with the same interval choices simultaneously) readily
yields that the map
\begin{align*}
\mathscr{L}_{m}^{n} &\longmapsto (\mathbb{R}^{m})^{n-1}   \\
(\boldsymbol{x}_{0},\boldsymbol{x}_{1},\ldots,\boldsymbol{x}_{n-1}) &\longmapsto (\boldsymbol{x}_{0}+\boldsymbol{x}_{1}, \boldsymbol{x}_{0}+\boldsymbol{x}_{2},\ldots,\boldsymbol{x}_{0}+\boldsymbol{x}_{n-1}),
\end{align*}
is surjective and therefore $\dim(\mathscr{L}_{m}^{n})\geq m(n-1)$.
Again
Theorem~\ref{uppack} gives the reverse estimate by using a consequence of a well-known
result by Jarn\'ik~\cite{jarnik} that $\dim(\mathscr{L}_{m})=0$.

\section{Sumsets and Cartesian products of sets of Diophantine numbers with restricted irrationality exponents}  \label{ch3}

\subsection{Definitions}
In this section we are concerned with Cartesian products of sets
of numbers which are approximable up to a given order by rational numbers.
For a real number $\xi$, we consider its irrationality exponent
$\mu(\xi)$, defined as the supremum 
of numbers $\mu$ for which the inequality
\begin{equation}  \label{eq:restr}
| \xi-\frac{p}{q}| \leq q^{-\mu}
\end{equation}
has infinitely many solutions in rational numbers $p/q$.
Then $\mu(\xi)=1$ for $\xi\in\mathbb{Q}$, and
by the theory of
continued fractions (or Dirichlet's theorem), $\mu(\xi)\geq 2$ for all $\xi\in \mathbb{R}\setminus \mathbb{Q}$.
Liouville numbers are precisely those $\xi$ with
$\mu(\xi)=\infty$, the complement is sometimes referred
to as Diophantine numbers. Further
define $\theta_{b}(\xi)$ like $\mu(\xi)$ above but where we restrict the 
approximating rationals $p/q$ in \eqref{eq:restr} to 
integral powers $q=b^{N}$
of some fixed integer base $b\geq 2$. 
This corresponds to $v_{b}(\xi)+1$
with exponent $v_{b}$ as defined in~\cite{ab10}. We also remark that the
exponent $\mu(\xi)$ was denoted by $v_1(\xi)+1$ in~\cite{ab10}. 
Clearly $\mu(\xi)\geq \sup_{b\geq 2} \theta_b(\xi)$.
Moreover, $\theta_{b}(\xi)\geq 1$ for
any $\xi\in\mathbb{R}$ and $b\geq 2$, with equality if $\xi\in \mathbb{Q}$.
We define level sets for both exponents.

\begin{definition}  \label{de}
	Let
	\[
	\mathscr{W}_{\lambda;\mu}= \{ \xi\in\mathbb{R}: \lambda\leq \mu(\xi)\leq \mu\}, \qquad 2\leq \lambda\leq \mu\leq \infty,
	\]
	and let
	\[
	\mathscr{W}_{\lambda}= \mathscr{W}_{\lambda;\lambda}= \{ \xi\in\mathbb{R}: \mu(\xi)=\lambda\}, \qquad \lambda\in[2,\infty].
	\]
	Further for $1\leq \lambda\leq \mu\leq \infty$ define
	$\mathscr{V}_{\lambda;\mu,(b)}$
	and $\mathscr{V}_{\lambda,(b)}$ accordingly
	with respect to the exponent $\theta_{b}(\xi)$ in place of $\mu(\xi)$.
\end{definition}

Any number
in $\mathscr{V}_{\lambda;\mu,(b)}$ with $\lambda>1$ has
arbitrarily long consecutive $0$ and/or $(b-1)$
digit strings in its base $b$ expansion.
Clearly the sets $\mathscr{W}_{\lambda;\mu}$ and $\mathscr{V}_{\lambda;\mu,(b)}$ become larger as $\lambda$
decreases and as $\mu$ increases. Moreover 
$\mathscr{V}_{\lambda;\mu,(b)}\subseteq \mathscr{W}_{\lambda;\infty}$
for $\mu\geq \lambda$.
However, notice that  $\mathscr{V}_{\lambda;\mu,(b)}\not\subseteq\mathscr{W}_{\lambda;\mu}$.
The set $\mathscr{W}_{\infty}$ is nothing but
the set of Liouville numbers treated in \S~\ref{1}.
The union of the sets $\mathscr{W}_{\lambda}$
over $\lambda>2$, that is the set of all numbers with $\mu(\xi)>2$,
is commonly referred to as the set of very well approximable numbers.

The remainder of the paper is driven by the following two questions that extend the results
of Erd\H{o}s~\cite{erdos} and their consequences recalled in~\S~\ref{1}.

\begin{problem}
	What can we say (metrically) about sumsets $\mathscr{W}_{\lambda_0}+\mathscr{W}_{\lambda_1}$ and $\mathscr{V}_{\lambda_0,(b)}+\mathscr{V}_{\lambda_1,(b)}$ for given
	$\lambda_0, \lambda_1$? 
\end{problem}

\begin{problem}
	Determine the Hausdorff and packing dimensions of Cartesian
	product sets $\prod_{i=0}^{n-1} \mathscr{W}_{\lambda_{i}}$
	and $\prod_{i=0}^{n-1} \mathscr{V}_{\lambda_{i},(b)}$ for given
	real numbers $\lambda_0,\ldots,\lambda_{n-1}$.
\end{problem}

\subsection{Main new results}  \label{newr}
While we are mainly concerned with Cartesian products of the sets $\mathscr{W}_{\lambda;\mu}$, our first result on sets $\mathscr{V}_{\lambda,(b)}$
is more complete. 

\begin{theorem}  \label{vthm}
	Let $n\geq 2$ be an integer. 
	For any integer $b\geq 2$ and $\lambda_{0},\ldots,\lambda_{n-1}$ in $[1,\infty]$, we have
	\begin{equation} \label{eq:ganslessen}
	n-1+\frac{1}{\max_{0\leq i\leq n-1} \lambda_{i}} \; \geq \; \dim(\prod_{i=0}^{n-1} \mathscr{V}_{\lambda_{i},(b)}) \; \geq \; \max\left\{n-1,\; \sum_{i=0}^{n-1} \lambda_{i}^{-1}\right\}.
	\end{equation}
	In particular, if all $\lambda_{i}$ are large enough compared to $n$
	we have
	\begin{equation} \label{eq:unusuale}
	\dim(\prod_{i=0}^{n-1} \mathscr{V}_{\lambda_{i},(b)})> \sum_{i=0}^{n-1}\dim(\mathscr{V}_{\lambda_{i},(b)}),
	\end{equation}
	and for every $n\geq 1$ we have
	\begin{equation} \label{eq:nmine}
	\lim_{\max \lambda_{i}\to\infty} \dim(\prod_{i=0}^{n-1} \mathscr{V}_{\lambda_{i},(b)})= n-1, \qquad
	\dim_{P}(\prod_{i=0}^{n-1} \mathscr{V}_{\lambda_{i},(b)})= n,
	\end{equation}
	where the limit is taken over any point $(\lambda_{0},\ldots,\lambda_{n-1})$
	whose maximum tends to infinity. 
\end{theorem}

The crucial claim is the lower bound $n-1$ in \eqref{eq:ganslessen}. 
In fact by
Borosh and Fraenkel~\cite{bf72} (see also Amou and Bugeaud~\cite{ab10}), we have the Hausdorff dimension formula
\begin{equation} \label{eq:vsets}
\dim(\mathscr{V}_{\lambda;\mu,(b)})= \frac{1}{\lambda}, \qquad \mu\geq \lambda\geq 1.
\end{equation}
Combined with \eqref{eq:most} and Theorem~\ref{uppack}, 
we then deduce all other assertions of \eqref{eq:ganslessen}.
From \eqref{eq:ganslessen}, \eqref{eq:vsets} we further derive \eqref{eq:unusuale} and \eqref{eq:nmine}. As to the packing dimension,
from Theorem~\ref{uppack} upon introducing another variable $\lambda_n=\infty$, we get
\[
n\ge \dim_{P}(\prod_{i=0}^{n-1} \mathscr{V}_{\lambda_{i},(b)})\ge 
\dim(\prod_{i=0}^{n} \mathscr{V}_{\lambda_{i},(b)}) -\dim(\mathscr{V}_{\lambda_n,(b)})\ge n-0=n,
\]
where we used \eqref{eq:ganslessen} in dimension $n+1$ 
and \eqref{eq:vsets} in the last inequality. We notice
that the case $n=2$ constitutes the fact
	for any integer $b\ge 2$ and any $\lambda\ge 1$, we have
	\begin{equation} \label{eq:naloger}
	\dim_P(\mathscr{V}_{\lambda,(b)})=1.
	\end{equation}
This formula may not come as a surprise to experts (see for example
\eqref{eq:MAR} below), however we believe it is a new result
worth being mentioned.

We formulate a conjecture corresponding to Theorem~\ref{vthm} 
for the sets with unrestricted 
rationals.

\begin{conjecture} \label{kond}
	For any $\lambda_{0},\ldots,\lambda_{n-1}$ in $[2,\infty]$, 
	we have
	\begin{equation} \label{eq:holdja}
	n-1+\frac{2}{\max_{0\leq i\leq n-1} \lambda_{i}} \; \geq \; \dim(\prod_{i=0}^{n-1} \mathscr{W}_{\lambda_{i}}) \; \geq \; \max\left\{ n-1,\; 2\sum_{i=1}^{n} \lambda_{i}^{-1}\right\}.
	\end{equation}
	In particular, if all $\lambda_{i}$ are large enough compared to $n$
	we have
	\[
	\dim(\prod_{i=0}^{n-1} \mathscr{W}_{\lambda_{i}})> \sum_{i=0}^{n-1}\dim(\mathscr{W}_{\lambda_{i}}),
	\]
	and for every $n\geq 1$, with the limit understood as in Theorem~\ref{vthm}, we have
	\[
	\lim_{\max \lambda_{i}\to\infty} \dim(\prod_{i=0}^{n-1} \mathscr{W}_{\lambda_{i}})= n-1, \qquad 
	\dim_P(\prod_{i=0}^{n-1} \mathscr{W}_{\lambda_{i}})= n.
	\]
\end{conjecture}

Unfortunately, as remarked above there is no inclusion between sets
$\mathscr{W}_{\lambda}$ and $\mathscr{V}_{\lambda,(b)}$
that would imply the claims via Theorem~\ref{vthm}.
The validity of the lower bound $n-1$ in \eqref{eq:holdja} is again the key problem.
Similar to the remarks below Theorem~\ref{vthm},
the remaining claims would again follow
via the special case $n=1$ of Jarn\'ik's formula~\cite{jarnik}:
\begin{equation} \label{eq:jarnik}
\dim(\mathscr{W}_{\lambda;\mu})=\frac{2}{\lambda}, \qquad\qquad 2\leq \lambda\leq \mu\leq \infty.
\end{equation}
Since for $\mu<\infty$ 
the sets $\mathscr{W}_{\lambda;\mu}$ in question are of first category as they lie in the complement
of the set of Liouville numbers $\mathscr{L}=\mathscr{W}_{\infty}$, 
we cannot apply topological arguments similar to the unconstructive
proof of $\mathscr{L}+\mathscr{L}=\mathbb{R}$ by Erd\H{o}s~\cite{erdos} recalled
in \S~\ref{1}. 
Indeed all proofs of partial results
below have constructive character, and rely on similar ideas as 
Erd\H{o}s' digit based proof explained in \S~\ref{1}. 
For completeness we also state the analogue of \eqref{eq:naloger} for unrestricted rationals which reads
\begin{equation}  \label{eq:MAR}
\dim_P(\mathscr{W}_{\lambda;\mu})= 1, \qquad\qquad 2\leq \lambda\leq \mu\leq \infty.
\end{equation}
This is a direct consequence of more general results by Marnat~\cite{marnat}. We refer to Theorem~\ref{ttt} 
and Corollary~\ref{letz} below for variants
of \eqref{eq:MAR} not implied by~\cite{marnat}.

Our first result supporting Conjecture~\ref{kond}
is that similar to Liouville numbers, certain
product sets of
$\mathscr{W}_{\lambda_{i};\mu_{i}}$ 
have indeed Hausdorff dimension at least $n-1$. This 
is the main substance of Theorem~\ref{th03}, where we also
add other bounds for completeness.

\begin{theorem} \label{th03}
	Let $n\geq 2$ be an integer. 
	Let $\lambda_{0},\ldots,\lambda_{n-1}$ and $\mu_{0},\ldots,\mu_{n-1}$ be real numbers or infinity
	satisfying $2\leq \lambda_i\leq \mu_i\leq \infty$. Suppose
	\begin{equation} \label{eq:happ}
	\mu_{i}>\frac{\Lambda}{\lambda_{i}-1}+1, \qquad\quad 0\leq i\leq n-1,
	\end{equation}
	where $\Lambda=\lambda_{0}\lambda_{1}\cdots \lambda_{n-1}$. Then we have
	\begin{equation} \label{eq:gans}
	n-1+\frac{2}{\max_{0\leq i\leq n-1} \lambda_{i}}\geq \dim(\prod_{i=0}^{n-1} \mathscr{W}_{\lambda_{i};\mu_{i}}) \geq \max\left\{n-1,\; 2\sum_{i=0}^{n-1} \lambda_{i}^{-1}\right\}.
	\end{equation}
\end{theorem}

If $\Lambda=\infty$ then the right hand sides in \eqref{eq:happ} are interpreted naturally in limits, and formal equalities $\infty=\infty$ suffice for the claim,
however this is of minor interest in view of 
Theorem~\ref{th01}. We emphasize that if otherwise 
$\Lambda<\infty$ (in fact one $\lambda_i=\infty$ is allowed) then all $\mu_{i}$ can be
effectively bounded, therefore the result is not covered
by Theorem~\ref{th01}. If some $\lambda_j$ is much larger than all other
$\lambda_i$, $i\neq j$, it may happen that the bound for 
$\mu_j$ in \eqref{eq:happ} is smaller than $\lambda_j$.
Then for this index we can take $\mu_j=\lambda_j$ and thus consider the set
$\mathscr{W}_{\lambda_j}$ replacing $\mathscr{W}_{\lambda_{j};\mu_{j}}$ in the Cartesian product. 
Theorem~\ref{th03}, as well as Theorem~\ref{th04}
below, contradicts the conjectured equality in~\cite[Conjecture~2.5]{ah18}
of the author, therefore the conditional implication in~\cite[Corollary~2.6]{ah18} 
is very open. 

We conclude from Theorem~\ref{th03} that for 
Cartesian products of general
sets $\mathscr{W}_{\lambda;\mu}$
there is no equality in \eqref{eq:most}.
For simplicity we restrict to all $\lambda_{i}$ being equal. 

\begin{corollary}  \label{C2}
	For $n\geq 2$ an integer. Let $\lambda,\mu$ be real numbers. If 
	$\lambda>2n/(n-1)$ and $\mu>(\lambda^{n}+\lambda-1)/(\lambda-1)$,
	we have
	\[
	\dim(\mathscr{W}_{\lambda;\mu}^{n}) > n\cdot \dim(\mathscr{W}_{\lambda;\mu}).
	\]
\end{corollary}

Another corollary to Theorem~\ref{th03} that contains
Theorem~\ref{th01} as a special case, with limits
understood as in Theorem~\ref{vthm} again, reads as follows. 

\begin{corollary}   \label{Ccor}
	Let $n,\lambda_{i},\mu_{i}$ be as in Theorem~\ref{th03}.
 Then
\[
\lim_{\max \lambda_{i}\to\infty} \dim(\mathscr{W}_{\lambda_{i};\mu_{i}}^{n})= n-1. 
\]
If we assume $\mu_i> (\max_{0\le i\le n-1} \lambda_i)\cdot \Lambda/(\lambda_i-1)+1$ for
$0\le i\le n-1$, then
\[
\lim_{\max \lambda_{i}\to\infty} \dim_P(\prod_{i=0}^{n-1} \mathscr{W}_{\lambda_{i};\mu_{i}})= n.
\]
In particular, if real numbers $\mu\ge \lambda\ge 2$ satisfy $\mu>(\lambda^{n+1}+\lambda-1)/(\lambda-1)$, then
\begin{equation*} 
\lim_{\lambda\to\infty} \dim_P(\mathscr{W}_{\lambda;\mu}^{n}) = n.
\end{equation*}
\end{corollary}

\begin{proof}
	The first claim is a direct consequence of Theorem~\ref{th03}.
	For the second, we introduce another variable $\lambda_n= \max_{0\le i\le n-1} \lambda_i$
 and let $\mu_n=\infty$. Then by assumption \eqref{eq:happ} and thus \eqref{eq:gans}
	hold in dimension $n+1$ for $\lambda_0,\ldots,\lambda_n$, and hence
	by Theorem~\ref{uppack}
	\[
	n\ge \dim_P(\prod_{i=0}^{n-1} \mathscr{W}_{\lambda_{i};\mu_{i}})\ge
	\dim(\prod_{i=0}^{n} \mathscr{W}_{\lambda_{i};\mu_{i}}) 
	- \dim(\mathscr{W}_{\lambda_{n};\mu_{n}})\ge n- \frac{2}{\lambda_n}=
	n- \frac{2}{\max_{1\leq i\leq n-1} \lambda_i }.
	\]
	The claims follow, and the last assertion is just the special case $\lambda:=\lambda_0=\cdots=\lambda_{n-1}$.
	\end{proof}

To continue with our second result 
towards Conjecture~\ref{kond}, we restrict ourselves to $n=2$ and consider the sets of precise order
of approximation. For the sequel fix 
\[
\rho:= \frac{5+\sqrt{17}}{2}=4.5615\ldots.
\]

\begin{theorem} \label{th04}
	Let $\lambda_{0}, \lambda_{1}$ be real numbers (or infinity) satisfying
	\begin{equation}  \label{eq:assuan}
	\min\{\lambda_{0}, \lambda_{1}\} > \rho.
	\end{equation}
	Then we have
	\[
	1+\frac{2}{\max \{\lambda_{0},\lambda_{1}\}}\geq \dim(\mathscr{W}_{\lambda_{0}}\times \mathscr{W}_{\lambda_{1}}) \geq 1= \max\{1,\; \frac{2}{\lambda_{0}}+\frac{2}{\lambda_{1}}\}.
	\]
	In particular,
	for every $\lambda\in (\rho, \infty]$ we have
	\[
	\dim(\mathscr{W}_{\lambda}\times \mathscr{W}_{\lambda}) \geq 1>2\dim(\mathscr{W}_{\lambda}).
	\]
\end{theorem}

For $n=2$ and $\lambda_{0}=\lambda_{1}=\lambda$, Conjecture~\ref{kond} remains open only for $\lambda\in (4,4.5615\ldots]$. 
Unfortunately, our method fails when $n\geq 3$.
Still Theorem~\ref{th04} admits the conclusion
that the Hausdorff dimension of any $n$-fold Cartesian product is comparable with $n$. 

\begin{corollary}
    Let $n\ge 1$ be an integer and $\lambda_0,\ldots,\lambda_{n-1}\in[2,\infty]$. Then we have  
    \begin{equation}  \label{eq:WIT}
\dim( \prod_{i=0}^{n-1} \mathscr{W}_{\lambda_{i}}) 
\ge \frac{2}{\rho}\cdot (n-1), \qquad \dim_P( \prod_{i=0}^{n-1} \mathscr{W}_{\lambda_{i}}) 
\ge \frac{2}{\rho}\cdot (n-2)+1.
\end{equation}
If $\min \lambda_{i}>\rho$, then we have the stronger bounds
\begin{equation} \label{eq:TWI}
\dim( \prod_{i=0}^{n-1} \mathscr{W}_{\lambda_{i}}) 
\ge \lfloor n/2\rfloor, \qquad \dim_P( \prod_{i=0}^{n-1} \mathscr{W}_{\lambda_{i}}) 
\ge \lfloor (n+1)/2\rfloor.
\end{equation}
\end{corollary}

\begin{proof} The left inequality of \eqref{eq:TWI} follows from
Theorem~\ref{th04} via recursive
application of \eqref{eq:most} and Theorem~\ref{th04}.
The right is deduced from the left with aid of \eqref{eq:uppak} and \eqref{eq:MAR} via
\[
\dim_P( \prod_{i=0}^{n-1} \mathscr{W}_{\lambda_{i}}) 
\ge \dim( \prod_{i=0}^{n-2} \mathscr{W}_{\lambda_{i}}) 
+ \dim_P( \mathscr{W}_{\lambda_{n-1}}) \ge 
\left\lfloor\frac{n-1}{2}\right\rfloor +1 =  \left\lfloor\frac{n+1}{2}\right\rfloor.
\]
For \eqref{eq:WIT}, let us
partition $\lambda_0,\ldots,\lambda_{n-1}$ into pairs $(\lambda_i, \lambda_j)$ with exponents both larger
resp. both smaller than $\rho$, with possibly up to two remaining indices. By Theorem~\ref{th04} we can estimate the Hausdorff dimension of two-fold products of 
such pairs from below by $2$ resp. $4/\rho<2$. It is easily checked that the worst case is that there is precisely one $i$ with $\lambda_i>\rho$, thus no pair of the first kind exists. Then \eqref{eq:most} easily leads to the claimed left bound, the right is again inferred with \eqref{eq:uppak} and \eqref{eq:MAR} as above.
\end{proof}

The following corollary emphasizes that for large $\lambda_i$, 
a big improvement of the lower bound
$1$ in Theorem~\ref{th04} cannot be made.

\begin{corollary}  
	With the definitions analogous to Theorem~\ref{vthm}, we have 
	\[
	\lim_{\min \{\lambda_{0},\lambda_{1}\}\to\infty} \dim(\mathscr{W}_{\lambda_{0}}\times \mathscr{W}_{\lambda_{1}}) =1.
	\]
\end{corollary}

It would be desirable
to deduce the limit $2$ for the packing dimension, however
Theorem~\ref{th04} seems to be insufficient for this conclusion, even
if $\lambda_0=\lambda_1$. 
We remark that combination of Theorem~\ref{th04} and Theorem~\ref{uppack} leads to a new proof of formula \eqref{eq:MAR}, for $\lambda>\rho$.
We refer to Theorem~\ref{ttt} below for a considerable generalisation.
Theorem~\ref{th04} can easily be deduced from the following claims on sumsets.

\begin{theorem}  \label{nonarxiv}
Let $\lambda_{0}, \lambda_{1}$ be real numbers satisfying \eqref{eq:assuan}. Then
\begin{equation}  \label{eq:AAL}
\mathscr{W}_{\lambda_{0}}+\mathscr{W}_{\lambda_{1}}
\supseteq \mathscr{W}_2,
\end{equation}
i.e. the sumset $\mathscr{W}_{\lambda_{0}}+\mathscr{W}_{\lambda_{1}}$ contains
any irrational real number that is not very well approximable.
Further, writing
$\lambda=\min\{\lambda_{0}, \lambda_{1} \}$,
its complement satisfies
\begin{equation}  \label{eq:FF}
\dim((\mathscr{W}_{\lambda_{0}}+\mathscr{W}_{\lambda_{1}})^{c})\leq
 \frac{2(2\lambda-1)}{ \lambda^2-\lambda }<1.
\end{equation}
In particular, we have
\begin{equation}  \label{eq:itsug}
\lim_{\min \{\lambda_{0},\lambda_{1}\} \to\infty} \dim((\mathscr{W}_{\lambda_{0}}+\mathscr{W}_{\lambda_{1}})^{c})= 0.
\end{equation}
\end{theorem}

\begin{remark}  \label{remarke}
For any $\lambda_{0}, \lambda_{1}$ the sumset
$\mathscr{W}_{\lambda_{0}}+\mathscr{W}_{\lambda_{1}}$
has either full or zero Lebesgue measure, since
this is true for any set invariant under rational translation. See also~\cite{cmo}. However, it is in 
general very unclear which of the two cases occurs.
\end{remark}

The basic idea in the proof of Theorem~\ref{nonarxiv} 
is to consider the base $5$
expansion (can be replaced by any integer $b\geq 2$) 
of any not very well approximable $\xi$, and to suitably manipulate the digits
in a sequence of intervals to form two numbers that sum up to $\xi$ and have 
the desired irrationality exponents.
We notice that if $\lambda_{0}\neq \lambda_{1}$, then the sum set
$\mathscr{W}_{\lambda_{0}}+\mathscr{W}_{\lambda_{1}}$ 
does not contain
any rational number since $\mu(\xi)=\mu(p/q-\xi)$,
hence its complement is non-empty. Moreover, it follows from 
Petruska~\cite[Theorem~1]{petruska} that if $\lambda_0<\lambda_1=\infty$ then 
$\mathscr{W}_{\lambda_{0}}+\mathscr{W}_{\lambda_{1}}$ 
does not contain any so called
strong Liouville number (see~\cite{alni2}). Hence 
the sumset has uncountable complement.
Nevertheless, our result \eqref{eq:itsug} suggests 
that the complement is always
very small in a metrical sense.

\begin{conjecture}
	For any $\lambda_0, \lambda_1\in [2,\infty]$, the set 
	$(\mathscr{W}_{\lambda_{0}}+\mathscr{W}_{\lambda_{1}})^{c}$ 
	has Hausdorff dimension $0$.
\end{conjecture}

For completeness, we state
the following consequence of Theorem~\ref{nonarxiv}
first noticed by Chalebgwa and Morris, indeed
a special case of~\cite[Theorem~4.3]{cmo}. 

\begin{corollary}[Chalebgwa, Morris]  \label{covier}
    Let $\lambda_0, \lambda_1, \lambda_2, \lambda_3\in (\rho,\infty]$. Then
    \[
    \mathscr{W}_{\lambda_{0}}+\mathscr{W}_{\lambda_{1}}+
    \mathscr{W}_{\lambda_{2}}+\mathscr{W}_{\lambda_{3}}=
    \mathbb{R}.
    \]
\end{corollary}

\begin{proof}
By \eqref{eq:AAL} the sets
$A:=\mathscr{W}_{\lambda_{0}}+\mathscr{W}_{\lambda_{1}}$
and $B:=\mathscr{W}_{\lambda_{2}}+\mathscr{W}_{\lambda_{3}}$
both have full Lebesgue measure. Hence a short argument yields $A+B=\mathbb{R}$, see~\cite{cmo} for the details.
\end{proof}

Similar claims can be derived from the underlying argument, for example the mixed sum-product set $(\mathscr{W}_{\lambda_{0}}+\mathscr{W}_{\lambda_{1}})\cdot
(\mathscr{W}_{\lambda_{2}}+\mathscr{W}_{\lambda_{3}})$ 
equals $\mathbb{R}$ if $\lambda_0=\lambda_1$ 
or $\lambda_2=\lambda_3$, and $\mathbb{R}\setminus \{0\}$
otherwise. While two sets are insufficient for the
conclusion of Corollary~\ref{covier}
as soon as $\lambda_0\ne \lambda_1$ 
by the above remarks, we believe 
that three sets always suffice.

\begin{problem}
    Is it true that for any $\lambda_0, \lambda_1, \lambda_2$ all at least $2$ (or large enough) we have
    \[
    \mathscr{W}_{\lambda_{0}}+\mathscr{W}_{\lambda_{1}}+
    \mathscr{W}_{\lambda_{2}}=
    \mathbb{R}\; ?
    \]
\end{problem}

The answer is affirmative in the special case that \eqref{eq:assuan} holds and $\lambda_2=2$ (after  relabelling if necessary),
by the argument of the proof of Corollary~\ref{covier}.
Recalling another result by Erd\H{o}s from~\cite{erdos} 
that any real number can be written as a
product of two Liouville numbers, it is natural to ask if the
product sets of Diophantine numbers with prescribed irrationality 
exponents are typically large as well.

\begin{problem}
	Is there an analogue of Theorem~\ref{nonarxiv} for the product sets
	$\mathscr{W}_{\lambda_{0}}\cdot \mathscr{W}_{\lambda_{1}}$? 
\end{problem} 

Any such set has either $0$ or 
full Lebesgue measure, as in Remark~\ref{remarke}. 

\subsection{Generalisations}  \label{gener}

In this section, we state several extensions of  Theorem~\ref{th04} and
Theorem~\ref{nonarxiv} to more general settings.
    First, we point out that the proof of Theorem~\ref{nonarxiv} below 
    implies that
    the smaller sumsets $(\mathscr{W}_{\lambda_{0}}\cap \mathscr{V}_{\lambda_{0},(b)}) +(\mathscr{W}_{\lambda_{1}}\cap \mathscr{V}_{\lambda_{1},(b)})$ still contain the set $\mathscr{W}_2$ of not very well approximable numbers, and 
    consequently
	\begin{equation}  \label{eq:neufo}
	\dim((\mathscr{W}_{\lambda_{0}}\cap \mathscr{V}_{\lambda_{0},(b)}) \times (\mathscr{W}_{\lambda_{1}}\cap \mathscr{V}_{\lambda_{1},(b)}))\geq 1,
	\end{equation}
	upon assumption \eqref{eq:assuan}. 
     We derive a variant of \eqref{eq:naloger}
     by an analogous proof.
	
	\begin{corollary} \label{c3}
		For any integer $b\geq 2$ and any $\lambda>\rho$, the set $\mathscr{W}_{\lambda}\cap \mathscr{V}_{\lambda,(b)}$
		has packing dimension $1$.
	\end{corollary}
	
	For $\lambda>\rho$, Corollary~\ref{c3}
	refines the observation that 
	$\mathscr{W}_{\lambda}\cap \mathscr{V}_{\lambda,(b)}$ is uncountable
	for any $\lambda>2$ due to Amou and Bugeaud~\cite[Theorem~5]{ab10}.

Next, certain variants of Theorem~\ref{th04} and Theorem~\ref{nonarxiv} can be proved. We first
present an application to the classical Cantor sets $C_{b,W}$ that
have already been discussed in \S~\ref{1}.

\begin{theorem} \label{PREIS}
	Assume $\lambda_{0}, \lambda_{1}$ satisfy \eqref{eq:assuan}. Let
	$K=C_{b,W}$ with $b\geq 3$. If $0\in W$, then
	\begin{equation}  \label{eq:friederich}
	(\mathscr{W}_{\lambda_{0}}\cap K) + (\mathscr{W}_{\lambda_{1}}\cap K)\supseteq \mathscr{W}_2 \cap \mathscr{V}_{1,(b)}\cap K,
	\end{equation}
	i.e. the sumset contains any number in $K$ that is neither very well approximable in the usual, nor in the $b$-ary setting.
Regardless if $0\in W$ or not, as a result we have 
	\begin{equation} \label{eq:KEIN}
	\dim((\mathscr{W}_{\lambda_{0}}\cap K)\times (\mathscr{W}_{\lambda_{1}}\cap K)) \geq \dim(K)=\ \frac{\log |W|}{\log b}.
	\end{equation}
	Consequently, for any $\lambda>\rho$ we have
	\begin{equation} \label{eq:WS}
	\dim_{P}(\mathscr{W}_{\lambda}\cap K)= \frac{\log |W|}{\log b}= \dim_{P}(K).
	\end{equation}
\end{theorem}

Note that \eqref{eq:KEIN} generalises \eqref{eq:minorett}. It has been known that $\mathscr{W}_{\lambda}\cap K$ 
is uncountable for any $\lambda\geq 2$ due to Bugeaud~\cite{bug2008}
(see also~\cite{lsv},~\cite{ichmjcnt}), 
so \eqref{eq:WS} refines this claim for large parameters $\lambda$.
The outline of the proof of \eqref{eq:friederich} is the same 
as that of Theorem~\ref{th04}. However, 
there are some adjustments to be done, which in particular for our proof to work, force the
sets $\mathscr{V}_{1,(b)}$ to enter in the right hand side of the inclusion. The implications \eqref{eq:KEIN}, \eqref{eq:WS} 
use metrical results from~\cite{lsv, weiss}.
Therefore 
we present a detailed proof of Theorem~\ref{PREIS} in \S~\ref{22}.
We remark that despite intense investigation, the Hausdorff dimensions
of the sets $\mathscr{W}_{\lambda}\cap K$ remain unknown. For partial results and further
references we refer to Levesley, Salp and Velani~\cite{lsv}, Bugeaud and Durand~\cite{bugdur}
and the recent preprint by Yu~\cite{chinese}.

Next, we turn towards generalisations with respect 
to the order of approximation. By minor adjustments of our proofs below, we can replace
the sets $\mathscr{W}_{\lambda_i}$ in Theorem~\ref{th04} and
Theorem~\ref{nonarxiv}
by a larger class of sets on which we impose refined approximation 
conditions. Concretely, we consider sets $Exact(\Phi)$ in the following definition.

\begin{definition} \label{df}
For $\Phi: \mathbb{N}\to (0,\infty)$ any function and $C\in(0,1)$,
let $\mathscr{W}_{\Phi, C} \subseteq \mathbb{R}$ be the set of 
real numbers $\xi$ that satisfy the properties that 
\[
|\xi- \frac{p}{q}| \leq \Phi(q)
\] 
has infinitely many solutions $p/q$, whereas the estimate
\[
|\xi- \frac{p}{q}| > C\cdot \Phi(q)
\] 
holds for all except possibly finitely many rational numbers $p/q$.
Define the 
set of numbers with
``exact approximation of order $\Phi$'' associated to
$\Phi$ via
\[
Exact(\Phi)= \bigcap_{ C\in (0,1) } \mathscr{W}_{\Phi, C}.
\] 
\end{definition}

The Hausdorff dimension of $Exact(\Phi)$ was studied by Bugeaud in
a series of papers, starting with~\cite{buge}.
See also the forthcoming work~\cite{ben}.
 We now generalise Theorem~\ref{th04} and
Theorem~\ref{nonarxiv} at once and in particular determine 
their packing dimension.

\begin{theorem} \label{ttt}
	For $i=0, 1$, assume $\lambda_i$ is a real
	number satisfying \eqref{eq:assuan} and let
	 $\Phi_i$ be functions as in Definition~\ref{df} that 
	satisfy $\Phi_i(q)\leq q^{-\lambda_i}$ for all large $q\geq q_0$. 
	Then we have
	\begin{equation}  \label{eq:fritzemerz}
    Exact(\Phi_0) + Exact(\Phi_1) \supseteq \mathscr{W}_2.
    \end{equation}
    In particular, the sumset $Exact(\Phi_0) + Exact(\Phi_1)$ has full Lebesgue measure, and consequently for $\Phi=\Phi_i$ as above we have
    \begin{equation} \label{eq:phan}
    \dim_P(Exact(\Phi))=1.
    \end{equation}
\end{theorem}

According generalisations of Theorem~\ref{vthm} and Corollary~\ref{c3}
can be formulated as well.
On the other hand, if an according variant of Theorem~\ref{PREIS} 
holds is not quite clear, see the remarks 
below its proof at the
end of the paper. As indicated in~\S~\ref{newr}, formula \eqref{eq:phan} is 
independent from the results in~\cite{marnat}. 
Indeed, due to their deduction
from the variational principle~\cite{dfsu}, 
the proofs in~\cite{marnat} ideally admit
the full dimension result only for the larger sets $\mathscr{W}_{\Phi,C}$ with some fixed $C\in(0,1)$.

We finally study generalisations to simultaneous approximation.
We only explicitly state an extension of Theorem~\ref{nonarxiv},
however the other results of~\S~\ref{newr}
can be modified accordingly. Let $\mathscr{W}_{\lambda}^{(m)}$ be the set of
real vectors $\boldsymbol{\xi}=(\xi_1,\ldots,\xi_m)$ simultaneously approximable of order precisely $\lambda$, i.e. so that $\mu_m(\boldsymbol{\xi})=\lambda$ 
where $\mu_m(\boldsymbol{\xi})$ is the supremum of $\lambda$ so that $|\max_{1\le k\le m} \xi_k-p_k/q|< q^{-\lambda}$ infinitely often. 
Note that $\mathscr{W}_{\lambda}=\mathscr{W}_{\lambda}^{(1)}$.
Then $\mu_m(\boldsymbol{\xi})\geq 1+1/m$ for any irrational $\boldsymbol{\xi}\in\mathbb{R}^m$ by Dirichlet's Theorem,
and $\boldsymbol{\xi}$ is called very well approximable if the inequality is strict. As for $m=1$, the latter vectors form a 
Lebesgue nullset in $\mathbb{R}^m$, see~\cite{jarnik}.

\begin{theorem}  \label{multidim}
	There exists a decreasing sequence of real numbers 
	\[
	\frac{5+\sqrt{17}}{2}=\rho=\rho_1> \rho_2>\rho_3>\cdots
	\]
	with limit 
	$\gamma:=(3+\sqrt{5})/2=2.6180\ldots$ and so that $\min \{ \lambda_0, \lambda_1\}>\rho_m$ implies
	\[
	\mathscr{W}_{\lambda_{0}}^{(m)}+\mathscr{W}_{\lambda_{1}}^{(m)}
	\supseteq \mathscr{W}_{1+1/m}^{(m)},\qquad m\ge 1,
	\]
	i.e. the 
	sumset contains
	any irrational real $m$-vector that is not very well approximable.
	Further, writing
	$\lambda=\min\{\lambda_{0}, \lambda_{1} \}$,
	its complement satisfies
	\begin{equation}  \label{eq:FF00}
	\dim((\mathscr{W}_{\lambda_{0}}^{(m)}+\mathscr{W}_{\lambda_{1}}^{(m)})^{c})\leq
	\frac{(m+1)(2\lambda-1)}{ \lambda^2-\lambda }<m.
	\end{equation}
	In particular, we have
	\begin{equation}  \label{eq:itsug00}
	\lim_{\min \{\lambda_{0},\lambda_{1}\} \to\infty} \dim((\mathscr{W}_{\lambda_{0}}^{(m)}+\mathscr{W}_{\lambda_{1}}^{(m)})^{c})= 0.
	\end{equation}
\end{theorem}

The extension to the accordingly defined sets 
$Exact^{(m)}(\Phi_i)$, $i=0,1$, holds again, for functions satisfying $\Phi_i(t)=\Phi_i^{(m)}(t)<t^{-\rho_m-\epsilon}$ for some $\epsilon>0$
and $t\ge t_0$. 
By Theorem~\ref{uppack},
this implies 

\begin{corollary}  \label{letz}
	For $\Phi$ as above, the sets $Exact^{(m)}(\Phi)$ 
	have full packing dimension $m$. 
\end{corollary}

Similar to
the special case $m=1$ in \eqref{eq:phan},
this result is again not
covered by~\cite{marnat, dfsu}.
It may be compared with a non metrical
claim by Jarn\'ik~\cite[Satz~5]{jarnik} that 
$Exact^{(m)}(\Phi)$ 
is uncountable for any reasonable function $\Phi$. See also Akhunzhanov~\cite{azhu} and Akhunzhanov 
and Moshchevitin~\cite{akhmosh} for refinements.

\section{Proofs of results in~\S~\ref{newr}} \label{22}
The principal idea of the proofs from~\S~\ref{newr}
is similar to Theorem~\ref{th01}. Again we define very elementary
Lipschitz maps from the corresponding
product sets into an Euclidean space with codimension $1$, with large image.
The next elementary lemma guarantees that these images 
still have full Lebesgue measure relative to
the corresponding dimension.

\begin{lemma}  \label{qlemma}
	If $A_{1},\ldots,A_{k}$ are subsets of $\mathbb{R}^{n}$, then 
	$\dim(\prod (A_{i}\cup \mathbb{Q}))= \dim(\prod A_{i})$.
\end{lemma}

\begin{proof}
	We only show the claim for two factors, i.e.
	$\dim((A\cup \mathbb{Q}) \times (B\cup \mathbb{Q}))= \dim(A\times B)$.
	The general case works very similarly.
	Clearly $\dim((A\cup \mathbb{Q}) \times (B\cup \mathbb{Q}))\geq \dim(A\times B)$ by monotonicity of Hausdorff dimension. 
	For the reverse estimate,
	the difference set 
	$((A\cup \mathbb{Q}) \times (B\cup \mathbb{Q}))\setminus (A\times B)$
	is contained in $(A\times \mathbb{Q}) \cup (\mathbb{Q}\times B) \cup \mathbb{Q}^{2}$.
	Clearly $\dim(\mathbb{Q}^{2})=0$ and $A\times \mathbb{Q}$ and
	$B\times \mathbb{Q}$ are countable unions of translates of $A,B$
	respectively, thus their dimensions are bounded by $\dim(A)$ and $\dim(B)$, respectively. 
	Hence by \eqref{eq:most} their union has dimension $\max\{ \dim(A), \dim(B)\}\leq \dim(A)+\dim(B)\leq \dim(A\times B)$.
	Thus adding the parts $A\times \mathbb{Q}, \mathbb{Q}\times B$ and $\mathbb{Q}^{2}$ to $A\times B$ does not increase the Hausdorff dimension.
\end{proof}

Notice we only used the property that $\mathbb{Q}$ is countable 
in the proof.
We also use continued fractions in the proofs.
The next proposition recalls a relation between the growth of the denominators of the convergents and order of approximation.

\begin{proposition} \label{kp}
	Let $\xi$ be an irrational real number and denote
	by $p_{k}/q_{k}$ its continued fraction convergents. 
	Let $\tau_k$ be the real numbers defined by
	\[
	|\xi - \frac{p_{k}}{q_{k}}| = q_{k}^{-\tau_{k}}.
	\]
	Then $\tau_{k}\geq 2$ for any $k\geq 1$ and we have
	$q_{k+1} \asymp q_{k}^{\tau_{k}-1}$.
	In other words
	\[
	q_{k+1} \asymp |q_{k}\xi-p_{k}|^{-1}.
	\]
\end{proposition}

\begin{proof}
	It is well-known from the theory of continued fractions 
	(see Perron~\cite{perron}) that
	\[
	\frac{1}{2q_{k+1} } < \frac{1}{ q_k + q_{k+1} } < |q_k\xi-p_k|<\frac{1}{q_{k+1}}.
	\]
	Hence we get $q_{k+1}\asymp |q_k\xi-p_k|^{-1}$, 
	consequently $q_{k+1}\asymp q_k^{\tau_k-1}$.
\end{proof}

Complementary to Proposition~\ref{kp}, we 
require Legendre's Theorem on continued
fractions that tells us that every good approximating rational 
is a convergent, see Perron~\cite{perron}.

\begin{theorem}[Legendre]  \label{lege}
	If $\xi\in\mathbb{R}$ and $p/q$ is rational and satisfies $|p/q-\xi|< q^{-2}/2,$ then $p/q$
	is a convergent to $\xi$. 
\end{theorem}

We will first prove Theorem~\ref{th03}.
	
\begin{proof}[Proof of Theorem~\ref{th03}]
	Let $\lambda_{i}\geq 2$ for $1\leq i\leq n$ be fixed throughout.
		The upper bound in \eqref{eq:gans} follows from \eqref{eq:jarnik}
	and Theorem~\ref{uppack}. We prove the lower bound.
	As mentioned above, by \eqref{eq:most} and  Jarn\'ik's formula
	\eqref{eq:jarnik} we have
	\[
	\dim(\prod_{i=0}^{n-1} \mathscr{W}_{\lambda_{i};\mu_{i}}) 
	\geq 
	\sum_{i=0}^{n-1}\dim(\mathscr{W}_{\lambda_{i};\mu_{i}})=
	2\sum_{i=0}^{n-1} \lambda_{i}^{-1}.
	\]
	To derive the lower bound $n-1$,
	we follow a similar idea as in the proof of Theorem~\ref{th01} above.
	We show that for every $\lambda_{i},\mu_{i}$ as in the theorem, 
	the map
	\begin{align}
	\Psi_{1}:\; \prod_{i=0}^{n-1} (\mathscr{W}_{\lambda_{i};\mu_{i}}\cup \mathbb{Q}) &\longmapsto \mathbb{R}^{n-1},   \label{eq:psi1} \\
	(x_{0},x_{1},\ldots,x_{n-1}) &\longmapsto (x_{0}+x_{1}, x_{0}+x_{2},\ldots,x_{0}+x_{n-1}),  \nonumber
	\end{align}
	is surjective. Then by Proposition~\ref{pp}, the 
	domain set of $\Psi_1$ has Hausdorff dimension 
	at least $n-1$ as well, 
	and finally by Lemma~\ref{qlemma}
    when we remove $\mathbb{Q}$ from each factor in the domain of the map, the lower bound $n-1$ remains valid.
	
	\underline{Construction of the preimage}:
	We may restrict the image to real vectors within $[0,1)^{n-1}$, so
	let $(\xi_{1},\ldots,\xi_{n-1})\in [0,1)^{n-1}$ be arbitrary.
	Write
	\begin{equation*} 
	\xi_i= \sum_{\ell\ge 1 }  \frac{ d_{i,\ell} }{ 5^{\ell} }, \qquad 
	d_{i,\ell}\in \{ 0,1,\ldots,4 \},\; 1\leq i\leq n-1,
	\end{equation*}
	for their base $5$ expansions.
	We construct a preimage under $\Psi_{1}$
	in our product set
	for given $\lambda_{i},\mu_{i}$ in the theorem.
	Partition the positive integers in interval sets $(I_{j})_{j\geq 1}$
	according to the following recursion.
	Let $I_{0}=\{ 1,2\}$ and write $g_{0}=1$ and $h_{0}=2$ for the interval ends. Define
	\[
	I_{j}=\{g_{j},g_{j}+1,\ldots,h_{j}\}
	\]
	for $j\geq 1$, where $g_{j}, h_{j}$ are recursively given by
	\begin{equation} \label{eq:stets}
	g_{j}= h_{j-1}+1, \qquad h_{j}=\lceil \lambda_{i}h_{j-1}\rceil,
	\end{equation}
	where we take $i$ the residue class of $j$ modulo $n$ in the 
	representation system $\{0,1,\ldots,n-1\}$.
	Thereby we obtain
	\[
	\frac{h_{j}}{g_{j}}= \lambda_{i}+o(1), \qquad \frac{g_{j+1}}{g_{j}}= \lambda_{i}+o(1),
	\]
	as $j\to\infty$,
	with $i=i(j)$ as above, that is two consecutive right (and left) 
	interval endpoints
	roughly differ by a multiplicative factor among our
	numbers $\lambda_{0},\ldots,\lambda_{n-1}$ depending on the index.
	In case some $\lambda_i=\infty$, instead we let the according quotients
	$h_j/g_j$ tend to infinity as $j\to\infty$, and the arguments below remain valid. 
	We construct the base $5$ expansion of a real number
	\[
	x_{0}^{\ast}= \sum_{\ell\ge 1 }  \frac{ d_{0,\ell}^{\ast} }{ 5^{\ell} }, \qquad d_{0,\ell}^{\ast}\in\{ 0,1,2,3,4\},
	\]
	as follows. For $1\leq i\leq n-1$,
	let the digits $d_{0,\ell}^{\ast}$ equal the base $5$ digit of $\xi_{i}$ in intervals $I_{j}$ with $j\equiv i\bmod n$, 
	i.e. $d_{0,\ell}^{\ast}= d_{i,\ell}$ for $i\equiv j\bmod n$, 
	where $j$ is the index with $\ell\in I_j$.
	Finally, put the base $5$ digit zero for digits in
	$I_{j}$ with $j\equiv 0\bmod n$, i.e. $d_{0,\ell}^{\ast}= 0$
	if $\ell\in I_j$ with $j\equiv 0\bmod n$. Then $x_{0}^{\ast}$ is well-defined.
	It will be convenient to let $\xi_0:= x_0^{\ast}$.
	We slightly alter $x_0^{\ast}$ at the interval endpoints
	to derive another real number $x_0$ via
	\begin{equation}  \label{eq:x0stern}
	x_0= x_0^{\ast} - \sum_{k\geq 1} \frac{a_k}{5^{h_k}}, \qquad a_k\in \{0,1,2,3,4\},
	\end{equation}
	with concrete choice of $a_k$ to be defined below (in fact it suffices to consider $a_k\in \{0,1,2\}$, or 
	alternatively $a_k\in \{-1,0,1\}$). 
	Further define
	\[
	x_{i}:= \xi_{i}-x_{0}, \qquad\qquad 1\leq i\leq n-1.
	\] 
	Notice that for $1\leq i\leq n-1$ and given large $j$ 
	with $j\equiv i\bmod n$, 
	we may write
	\begin{align*}
	x_i&= \sum_{\ell\ge 1} \frac{ d_{i,\ell} }{ 5^{\ell} } - \sum_{\ell\ge 1} \frac{ d_{0,\ell}^{\ast} }{ 5^{\ell} } + \sum_{k\geq 1} \frac{a_k}{5^{h_k}}\\ &=
	\sum_{\ell=1}^{h_{j-1} } \frac{ d_{i,\ell}-d_{0,\ell}^{\ast} }{ 5^{\ell} } +  \sum_{k=1}^{j-1} \frac{a_k}{5^{h_k}} + \sum_{\ell\;>\; h_{j} } \frac{ d_{i,\ell}-d_{0,\ell}^{\ast} }{ 5^{\ell} }+
	\sum_{k\ge j} \frac{a_k}{5^{h_k}} 
	\end{align*}
	where we used that in the missing middle range at positions $\ell\in I_{j}=[h_{j-1}+1,h_{j}]\cap \mathbb{Z}$
	 the digits $d_{i,\ell}$ and $d_{0,\ell}^{\ast}$ are equal by construction, so the difference vanishes. For given $i,j$
	 with $i\equiv j\bmod n$,
	adding the first two finite sums gives a rational number
	\[
	\frac{p_{j}}{q_{j}}=\frac{p_{i,j}}{q_{i,j}}= \frac{ e_{i,j} }{ 5^{h_{j-1} } }, \qquad 
	e_{i,j}= \sum_{\ell=1}^{h_{j-1}} 5^{ h_{j-1}-\ell } (d_{i,\ell}-d_{i,\ell}^{\ast}) + \sum_{k=1}^{j-1} 5^{ h_{j-1}-h_k } a_k\in\mathbb{Z}.
	\]
	We emphasize that throughout we will assume $j\equiv i\bmod n$ whenever we just write $p_j/q_j$.
	We remark that the $p_{i,j}/q_{i,j}$ are essentially 
	obtained by chopping off the base $5$ expansion of $x_i$
	after position $h_{j-1}$.
	Notice that
	\[
	e_{i,j}\equiv d_{i,h_{j-1}} - d_{i,h_{j-1}}^{\ast} + a_{j-1}\bmod 5
	\] 
	obtained from the last terms of the two sums,
	since all other involved expressions
	are divisible by $5$. Thus it is possible to choose $a_{j-1}\in \{0,1,\ldots,4\}$ in each step so that $5\nmid e_{i,j}$ 
	by just avoiding the residue class $-(d_{i,h_{j-1}} - d_{i,h_{j-1}}^{\ast})$ modulo $5$.
	Then all $p_{i,j}/q_{i,j}=e_{i,j}/5^{h_{j-1}}$ are reduced
	(note that for different $i$ we consider different $a_{j-1}$ by the restriction $i\equiv j\bmod n$, so we have only one condition 
	for each $j$). The remainder term from
	the remaining two infinite sums can obviously be estimated by 
	\begin{equation}  \label{eq:FRi}
	|x_i- \frac{p_{i,j}}{q_{i,j}}|= \left| \sum_{\ell\;>\; h_{j} } \frac{ d_{i,\ell}-d_{0,\ell}^{\ast} }{ 5^{\ell} }+
	\sum_{k\ge j} \frac{a_k}{5^{h_k}} \right| \ll 5^{-h_{j}}.
	\end{equation}
	
	For $i=0$, there is a small twist. Since $d_{0,\ell}^{\ast}=0$
	if $\ell\in I_j$ with $j\equiv 0\bmod n$, we have for any $j\equiv 0\bmod n$ that
	\[
	x_0= \sum_{\ell=1}^{h_{j-1}} \frac{ d_{0,\ell}^{\ast} }{ 5^{\ell} }
	-  \sum_{k=1}^{j-1} \frac{a_k}{5^{h_k}} + \sum_{\ell\;>\;h_j} \frac{ d_{0,\ell}^{\ast} }{ 5^{\ell} }-
	\sum_{k\ge j} \frac{a_k}{5^{h_k}}. 
	\]
	Now again taking the rational number obtained from subtracting
	the second sum the first, we get rationals $p_{0,j}/q_{0,j}=e_{0,j}/5^{h_{j-1}}$. By an analogous argument
	excluding the residue class of $d_{0,h_{j-1}}^{\ast}\bmod 5$
	for $a_{j-1}$, we can assume that all $p_{0,j}/q_{0,j}=e_{0,j}/5^{h_{j-1}}$ 
	are already in reduced form.
	The analogue of the remainder term estimate \eqref{eq:FRi} holds
	for similar reasons for $i=0$ as well.
    
	To finish the proof,
	we need to show the following claim.
	
	\textbf{Claim:} The numbers
	 $x_{i}$ for $0\leq i\leq n-1$ all lie in the prescribed $\mathscr{W}_{\lambda_{i};\mu_{i}}\cup \mathbb{Q}$, i.e.
	 they take irrationality exponents in the 
	 corresponding intervals $[\lambda_i, \mu_i]$ unless they are rational.

	\underline{Proof of the claim}.
	Let $0\leq i\leq n-1$ be fixed. We easily check from the above that that $\mu(x_{i})\geq \lambda_{i}$. Indeed, considering
	$p_{i,j}/q_{i,j}$ constructed above for $i\equiv j\bmod n$, 
	from \eqref{eq:FRi} (which we observed
	also holds for $i=0$) and $q_{i,j}=5^{h_{j-1}}\asymp 5^{g_j}$ via
	 \begin{equation} \label{eq:STERN}
	|x_{i}-p_{i,j}/q_{i,j}|\ll 5^{-h_j}\ll 5^{-g_{j+1}}
	\ll  5^{-\lambda_{i}g_{j}}\ll q_{i,j}^{-\lambda_{i}}.
	\end{equation}
	This means $\mu(x_{i})\geq \lambda_{i}$ unless the approximations
	are ultimately constant and equal to $x_{i}$, thus $x_{i}\in\mathbb{Q}$.
	   
	We need to show the inequality
	$\mu(x_{i})\leq \Lambda/(\lambda_{i}-1)+1$ for $0\leq i\leq n-1$.
	Write $\nu_{i}=\Lambda/(\lambda_{i}-1)+1$ for simplicity.
	Assume the contrary that $\mu(x_{i})>\nu_{i}$ for some $i$.
	Then for some $\mu>\nu_{i}$ the inequality
	\begin{equation} \label{eq:property}
	|x_{i}-\frac{p}{q}| < q^{-\mu}
	\end{equation}
	has infinitely many rational solutions $p/q$. We consider $i$ 
	fixed in the sequel. 
	We distinguish two cases: $p/q$ can be among the $p_{j}/q_{j}$
	defined above,
	with the convention $i\equiv j\bmod n$ as explained above,
	or distinct from them.
	
	\underline{Case 1}: The rational in \eqref{eq:property}
	satisfies $p/q=p_{j}/q_{j}$ for some $j\equiv i\bmod n$.
	Since we have
	observed that $p_{j}/q_{j}$ with $p_{j}=e_{i,j}, q_{j}=5^{h_{j-1}}$ for some $j\equiv i\bmod n$
	are reduced, we can restrict to $p=e_{i,j}, q=5^{h_{j-1}}$.
	Hence it remains to be checked
	that the reverse inequality to \eqref{eq:STERN} holds, i.e.
	\begin{equation}  \label{eq:siehda}
	|x_{i}-\frac{p_{j}}{q_{j}}| \gg 5^{-h_{j}},
	\end{equation}
	for any $j\equiv i\bmod n$. 
	However, if $1\leq i\le n-1$, we have
	\[
	|x_{i}-\frac{p_{j}}{q_{j}}|= |  \sum_{\ell\;>\; h_{j} } \frac{ d_{i,\ell}-d_{0,\ell}^{\ast} }{ 5^{\ell} }+
	\sum_{k\ge j} \frac{a_k}{5^{h_k}}  |,
	\]
	and similarly for $i=0$. In any case,
	the first sum is fixed and the second has dominating term
	$a_j/5^{h_j}\asymp 5^{-h_j}$, so it is clear that at most one integer choice $a_{j}$ may contradict \eqref{eq:siehda}. 
	So far, 
	we have only excluded one residue class modulo $5$ above for any $a_k$. Thus, upon avoiding another residue class for $a_j$ if necessary, clearly a suitable choice of the $a_j\in\{ 0,1,2,3,4\}$ 
	remains possible.

	\underline{Case 2}: Now assume infinitely many $p/q$ 
	with property \eqref{eq:property} are distinct
	from all $p_{j}/q_{j}$, where we assume $j\equiv i\bmod n$. 
	Let $\lambda=\lambda_{i}$.
	First we settle that any such $p/q$ as
	in \eqref{eq:property} must satisfy
	\begin{equation} \label{eq:must}
	q_{u}^{\lambda-1}\ll q\ll q_{u+n}^{1/(\mu-1)},
	\end{equation}
	for some $u\equiv i\bmod n$. We can take $u$ 
	to be the unique integer satisfying
	$u\equiv i\bmod n$ and so that
	$q_{u}<q\leq q_{u+n}$.
	In \eqref{eq:STERN} we have noticed that
	$| x_{i}-\frac{p_{u}}{q_{u}}| \ll q_{u}^{-\lambda}$.
	Since $\lambda>2$, clearly the reduced fraction $p_{u}/q_{u}$ is a convergent in the continued fraction expansion of $x_{1}$ by Legendre's Theorem~\ref{lege}.
	Then Proposition~\ref{kp} gives
	that the next convergent denominator is $\gg q_{u}^{\lambda-1}$ and
	since $p/q$ is clearly also a convergent to $x_{i}$ 
	by \eqref{eq:property} and Legendre's Theorem and as $q>q_{u}$, we infer $q\gg q_{u}^{\lambda-1}$,
	the left estimate in \eqref{eq:must}. 
	The right is induced very similarly using the assumption 
	that $p/q$ satisfies $|x_{i}-p/q| < q^{-\mu}$. Indeed, by
	Proposition~\ref{kp}, the subsequent
	convergent has denominator $\gg q^{\mu-1}$. Since $p_{u+n}/q_{u+n}$
	is another reduced convergent to $x_{1}$ and $q_{u+n}\ge q$
	and $p_{u+n}/q_{u+n}\neq p/q$ by assumption of Case 2, we derive $q_{u+n}\gg q^{\mu-1}$, equivalent
	to the right bound in \eqref{eq:must}. 
	
	Now from \eqref{eq:must} we see that
	\[
	q_{u+n} \gg q_{u}^{(\lambda-1)(\mu-1)}.
	\]
	On the other hand \eqref{eq:stets} implies
	$q_{u+n}\asymp q_{u}^{\Lambda}$.
	Hence $\Lambda\geq (\lambda-1)(\mu-1)-\epsilon_{1}$, which by choice of
	$\mu>\nu_{i}$ is however false for $\epsilon_{1}$
	small enough in view of assumption \eqref{eq:happ}. 
	Thus we have derived the desired contradiction to \eqref{eq:property} in both cases
	and our claim is proved. We conclude that the dimension of our 
	Cartesian product set is at least $n-1$. 
\end{proof}

If all $\lambda_{i}=\infty$, we have already proved the claim within the proof of Theorem~\ref{th01}. 
We prove Theorem~\ref{vthm} in a similar way. 
A notable twist occurs in
Case 2. To rule out
the existence of putative good approximations of the form
$p/b^{N}$ different from the $p_{u}/q_{u}$, we cannot
use the argument above since we do not assume \eqref{eq:happ}
any longer. To find a way around this obstruction, 
we will employ the following easy lemma. 

\begin{lemma} \label{lemur}
	Let $b\geq 2$ be an integer. Let $\mathscr{T}\subseteq \mathbb{R}^{n-1}$ be 
	the set of real vectors $(\xi_{1},\ldots,\xi_{n-1})$ for which
	any $\xi_{i}$ for $1\leq i\leq n-1$, as well as any  
	$\xi_{i}-\xi_{t}$ for every index pair $1\leq i<t\leq n-1$, 
	all lie in $\mathscr{V}_{1,(b)}$. Then $\mathscr{T}$
	has full $(n-1)$-dimensional Lebesgue measure. 	In fact the complement set $\mathbb{R}^{n-1}\setminus \mathscr{T}$ 
	is of Hausdorff dimension $n-2$. 
\end{lemma}

The proof is not deep and 
only uses standard measure theoretic arguments.

\begin{proof}
	It is well-known and follows from \eqref{eq:vsets} 
	by a standard measure theoretic
	argument that the complement of $\mathscr{V}_{1,(b)}$ in $\mathbb{R}$ has Lebesgue measure $0$ (is a nullset in $\mathbb{R}$). 
	We can write the set $\mathscr{T}$ in the lemma
	as the intersection of $U=\mathscr{V}_{1,(b)}^{n-1}$
	with the sets 
	\[
	U_{i,t}:= \{ (\xi_{1},\ldots,\xi_{n-1})\in\mathbb{R}^{n-1}:\;  
	\xi_{i}-\xi_{t}\in \mathscr{V}_{1,(b)}\}, \qquad 1\leq i<t\leq n-1.
	\]
	Clearly $U$ has full $(n-1)$-dimensional Lebesgue measure
	as it is the Cartesian product of full measure sets. 
	We show that
	every set $U_{i,t}$ has full measure as well.
	Indeed, upon relabelling, it can be identified 
	with $\mathbb{R}^{n-3} \times V$ where $V\subseteq \mathbb{R}^{2}$
	is given by $(\mathscr{V}_{1,(b)}+\mathbb{R}) \times \mathbb{R}=\{ (x+y,y): x\in \mathscr{V}_{1,(b)}, y\in\mathbb{R} \}$. 
	Hence the set $V$
	is the image of $\mathscr{V}_{1,(b)}\times \mathbb{R}$ 
	under the Lipschitz map
	$(x,y)\to (x+y,y)$. Since $\mathscr{V}_{1,(b)}\times \mathbb{R}\subseteq \mathbb{R}^{2}$ 
	has full $2$-dimensional Lebesgue measure again as a 
	Cartesian product of full measure sets,
	the same applies to $V$ by Proposition~\ref{pp}. 
	Again using that it can be written as Cartesian product of full measure
	sets, we obtain that any set 
	$U_{i,t}$ has full $(n-1)$-dimensional
	Hausdorff measure. Hence
	the intersection of the finitely many $U_{i,t}$
	has full measure as well, and intersecting it with the full measure set
	$U$ again preserves the property.
\end{proof}

\begin{proof}[Proof of Theorem~\ref{vthm}]
	Let $b\geq 2$ be any integer.
	The main claim is the lower bound $n-1$ in \eqref{eq:ganslessen}.
	We proceed as in the proof of Theorem~\ref{th03}, where for obvious reasons
	we work in base $b$ instead of $5$ (very minor adaptions
	related to the choice of the $a_j$  have to be made when $b=2$).
	We show that for any choices of $\lambda_{i}\geq 1$,
	the image of the map
	\begin{align}
	\Psi_{2}:\; \prod_{i=0}^{n-1} (\mathscr{V}_{\lambda_{i},(b)} \cup \mathbb{Q}) &\longmapsto \mathbb{R}^{n-1}, \label{eq:psi2}  \\
	(x_{0},x_{1},\ldots,x_{n-1}) &\longmapsto (x_{0}+x_{1}, x_{0}+x_{2},\ldots,x_{0}+x_{n-1}),  \nonumber
	\end{align}
	contains the set $\mathscr{T}$ from Lemma~\ref{lemur}.
	Since the lemma claims that $\mathscr{T}$ has full
	$(n-1)$-dimensional measure,
	then again we can conclude with Proposition~\ref{pp} and Lemma~\ref{qlemma}.
	
	We start with arbitrary $\boldsymbol{\xi}\in \mathscr{T}$
	and derive the preimage components $x_{i}$ by the same construction as in the proof of Theorem~\ref{th03}.  
	Define the rational approximations $p_{j}/q_{j}$ to $x_{i}$
	identically as in Theorem~\ref{th03} as well.
	They are of the desired form $p/b^{N}$, with $p=p_j$ and $N=h_{j-1}$, and approximate $x_i$
	of order $\lambda_{i}$, hence
	$\theta_{b}(x_{i})\geq \lambda_{i}$, unless $x_{i}\in\mathbb{Q}$. We have to prove the reverse
	estimate $\theta_{b}(x_{i})\leq \lambda_{i}$. Notice that in contrast to Theorem~\ref{th03} we do not impose
	the assumption \eqref{eq:happ} here. However, it will follow from our setup 
	and $\boldsymbol{\xi}\in\mathscr{T}$ 
	that no sufficiently good rational approximations with
	denominator an integer power of $b$ to $x_i$
	can exist. We again distinguish the two 
	cases $p/q=p_j/q_j$ and $p/q$ not of this form.
	The proof of Case 1 is immediate
	as in Theorem~\ref{th03}, 
	however it is also implicitly contained 
	in the general argument below. So again assume otherwise 
	that $x_i$ satisfies $\theta_b(x_i)>\lambda_i$ for
	some $i\in\{0,1,\ldots,n-1\}$ that we assume fixed. 
	Then the estimate
	\begin{equation}  \label{eq:appwell}
	|x_i- \frac{p}{b^N}| \leq b^{-\mu N} < b^{-\lambda_i N}
	\end{equation}
	holds for some $\mu>\lambda_i$ and infinitely many 
	pairs of integers $p, N$.
	For the moment, let us fix such large $p,N$.
	Then $x_i$ has a string of digits all equal to $0$ 
	or all equal to $b-1$ at positions
	in the interval $J=\{ N+1,N+2,\ldots,\lfloor \mu N\rfloor \}$. 
	Note that if $N$ is large then $J$ defined above cannot be contained in some $I_j$ with $j\equiv i \bmod n$ because $\max I_j/ \min I_j= h_j/g_j= \lambda_i+o(1)$
	as $j\to\infty$ by construction and $\mu=\max J/\min J-o(1)>\lambda_i$. More precisely,
	if we let $\epsilon=(\mu-\lambda_i)/2>0$, it is clear that $J$ intersects some $I_{j_{0}}$ with $j_0\not\equiv i\bmod n$ in some interval $J^{\prime}=\{ M,M+1,\ldots,\lfloor M(1+\epsilon)\rfloor\}\subseteq J$.
	Consider first $i=0$. Recall that by construction of $x_0$, 
	it has locally the same digits as some $\xi_j$ for some $1\le j\le n-1$ within $J^{\prime}$. Hence for $j\in\{ 1,2,\ldots,n-1\}$ the residue class of $j_0$ above modulo $n$, the real number $\xi_{j}$ has a pure $0$ 
	or $b-1$ digit
	string in $J^{\prime}$. However,
	this means 
	\[
	|\xi_{j}- \frac{p_0}{b^M}| < b^{-(1+\epsilon)M}
	\]
	for some integer $p_0$ and thus
	contradicts $\xi_{j}\in \mathscr{V}_{1,(b)}$.
	Let us now consider $i\ne 0$. 
	Then, since $x_0$ has the same digits as $\xi_j$ for again $j\equiv j_0\bmod n$ within $J^{\prime}\subseteq I_{j_0}$, 
	the fact that $x_i= x_0-\xi_i$ has zero digits in $J^{\prime}$ means
	that we may write
	\[
	x_i= \frac{p_1}{b^{M-1}} + (\xi_j-\xi_i) + E, \qquad |E|\ll b^{- M(1+\epsilon)} 
	\]
	for an integer $p_1$, coming from certain base $b$ digit
	differences between certain $\xi_t$, at the positions up to $M-1$.
	But this means
	\[
	| (\xi_j- \xi_i) -x_i + \frac{p_1}{b^{M-1}} | \ll b^{ - M(1+\epsilon) }. 
	\]
	On the other hand, estimate \eqref{eq:appwell} and $J^{\prime}\subseteq J$
	imply that we may write $x_i= p_2/b^M+O(b^{ - M(1+\epsilon) })$ for an integer $p_2$.
	Combining and by triangle inequality we derive
	\[
	| (\xi_j- \xi_i) -\frac{p_2}{b^M} + \frac{p_1}{b^{M-1}} |=
	| (\xi_j- \xi_i) -\frac{p_3}{b^M}|
	 \ll b^{ -  M(1+\epsilon)}
	\]
	for some integer $p_3$, which finally again contradicts
	our assumption $\boldsymbol{\xi}\in \mathscr{T}$.
	Hence the hypothetical assumption $\theta_b(x_i)>\lambda_i$
	cannot occur,
	and the proof of the lower bound $n-1$ for the Hausdorff dimension
	is finished.
	The remaining claims follow easily when taking \eqref{eq:most}
	and \eqref{eq:vsets} into account, as explained below 
	the formulation of Theorem~\ref{vthm}.
\end{proof}

We remark that for $n=2$ and $\lambda_{0}\neq \lambda_{1}$, 
the map $\Psi_{2}$ when restricted to
$\prod_{i=0}^{n-1} \mathscr{V}_{\lambda_{i},(b)}= \mathscr{V}_{\lambda_{0},(b)}\times \mathscr{V}_{\lambda_{1},(b)}$
is not surjective.
Indeed, any rational of the form $p/b^{N}$
cannot be in the image
of $\Psi_{2}$ because $\theta_{b}(\xi)=\theta_{b}(p/b^{N}-\xi)$ is easily verified. 

The proof of Theorem~\ref{nonarxiv} and consequently
Theorem~\ref{th04} again
uses similar ideas as Theorem~\ref{th03}. We first show how Theorem~\ref{th04} can be inferred
from Theorem~\ref{nonarxiv}.  
Here we consider not very well approximable numbers in the 
one-dimensional image of our sum map.

\begin{proof}[Deduction of Theorem~\ref{th04} from Theorem~\ref{nonarxiv}]
	Anything apart from the 
	lower bound $1$
	upon our assumption \eqref{eq:assuan} on the $\lambda_{i}$,
	follows easily from  \eqref{eq:jarnik} and Theorem~\ref{uppack}
	again. 
	We can assume both $\lambda_{i}<\infty$, 
	otherwise the claim follows directly from the upper bound being $1$.
	However, by Theorem~\ref{nonarxiv},
	the image of the Lipschitz map
	\begin{align} 
	\Psi_{3}:\; (\mathscr{W}_{\lambda_{0}} \cup \mathbb{Q})\times (\mathscr{W}_{\lambda_{1}}\cup \mathbb{Q}) &\longmapsto \mathbb{R},  \label{eq:psi3}  \\
	(x_{0},x_{1}) &\longmapsto x_{0}+x_{1},   \nonumber
	\end{align}
	contains the set
	$\mathscr{W}_2$ of not very well approximable numbers. 
	As $\mathscr{W}_2$ has full $1$-dimensional Lebesgue
	measure, the claim follows from Proposition~\ref{pp}
	and Lemma~\ref{qlemma}. 
	\end{proof}

Now we turn towards the proof of Theorem~\ref{nonarxiv}.
We again use a similar construction as in the proof of 
Theorem~\ref{th03}. The core of the proof is
that our restriction to $\xi$ that are not very well approximable 
in \eqref{eq:AAL}
will guarantee that there is no other good rational approximation to $x_{i}, i=0,1$
apart from the $p_j/q_j$ constructed in the proof of the lower bound.  Indeed, with some trick (that only works for $n=2$), we show that the existence of other putative good approximations would induce good rational approximations to $\xi$, which contradicts the hypothesis $\xi\in \mathscr{W}_2$. 
	
	\begin{proof}[Proof of Theorem~\ref{nonarxiv}]
	Start with arbitrary $\xi\in \mathscr{W}_2\cap [0,1)$ with base $5$
	representation
	\[
	\xi= \sum_{\ell\ge 1} \frac{d_{\ell}}{5^{\ell}}, \qquad d_{\ell}\in\{0,1,2,3,4\}. 
	\] 
	To show \eqref{eq:AAL},
	we construct
	$x_0, x_1$ that sum up to $\xi$ and have the prescribed
	irrationality exponents $\lambda_0$ and $\lambda_1$, respectively.
	 Similar to the proof of Theorem~\ref{th03}, partition $\mathbb{N}$ into
	intervals $I_{j}=\{g_{j},g_{j}+1,\ldots,h_{j}\}$ with 
	$g_{j+1}=h_{j}+1$ and $h_j=\lceil \lambda_1 g_j\rceil$ for odd $j$
	and $h_j=\lceil \lambda_0 g_j\rceil$ for even $j$. In particular
	$h_{j}/g_{j}= \lambda_{1}+o(1)$ for odd $j$
	and $h_{j}/g_{j}= \lambda_{0}+o(1)$ for even $j$, as $j\to\infty$.
	We further repeat the construction of $x_0, x_1$ from
	that proof. Due to $n=2$, the construction can be 
	described in a slightly simpler way here, which we want to explain.
	Let $x_{0}^{\ast}$ be the real number that has
	the base $5$ digits of $\xi$ in intervals $I_{j}$ for odd $j$ and $0$ in 
	intervals $I_{j}$ for even $j$, 
	which agrees with the definition of $x_0^{\ast}$ from the proof
	of Theorem~\ref{th03} when $n=2$.
	Let vice versa $x_{1}^{\ast}$ be the real number with 
	that has the digits of $\xi$ in intervals $I_j$ for even $j$
	and zero in $I_j$ for odd $j$, which corresponds to $\xi-x_0^{\ast}$.
	This means if we write the base $5$ expansions
		\[
	x_0^{\ast}= \sum_{\ell\ge 1} \frac{ e_{\ell}  }{5^{\ell}}, \qquad x_1^{\ast}= \sum_{\ell\ge 1} \frac{ f_{\ell}}{5^{\ell}},
	\] 
	then $e_{\ell}= d_{\ell}$ if $\ell\in I_j$ for odd $j$
	and $e_{\ell}=0$ else if $\ell\in I_j$ for even $j$, and
	vice versa for $f_j$.
	Let $a_k\in\{0,1,\ldots,4\}$ for $k\geq 1$ to be chosen later on. We again modify the digits
	at the interval endpoints by considering
	\[
	x_0= x_0^{\ast} - \sum_{k\geq 1} \frac{a_{k}}{5^{h_{k} }}, \qquad x_1= x_1^{\ast} + \sum_{k\geq 1} \frac{a_{k}}{5^{h_{k} }}.
	\]
	Then
	clearly $x_{0}+x_{1}=x_0^{\ast}+x_1^{\ast}=\xi$.
	We need to show that $\mu(x_i)=\lambda_i, i=0,1$ upon the lower bound assumed on $\lambda_i$.
	First observe that when again splitting the series for $x_0^{\ast}$ and $x_1^{\ast}$, for even $j\ge 1$ we may write
	\begin{equation*}  
	x_0= \sum_{\ell=1}^{h_{j-1}} \frac{ e_{\ell}}{5^\ell} -
	\sum_{k=1}^{j-1}  \frac{a_k}{5^{h_k}} + \sum_{\ell>h_{j}} \frac{ e_\ell}{5^{\ell}} -
	\sum_{k\geq j}  \frac{a_k}{5^{h_k}},
	\end{equation*}
	and for odd $j$ similarly we have
		\begin{equation}  \label{eq:REP2}
	x_1= \sum_{\ell=1}^{h_{j-1}} \frac{ f_{\ell}}{5^\ell} +
	\sum_{k=1}^{j-1}  \frac{a_k}{5^{h_k}} + \sum_{\ell>h_{j}} \frac{ f_\ell}{5^{\ell}} +
	\sum_{k\geq j}  \frac{a_k}{5^{h_k}}.
	\end{equation}
	Here again we used the vanishing of the digits in intervals $I_j=[h_{j-1}+1,h_j]$ for $j$ with the stated parity.
	We again consider the rational numbers $p_{i,j}/q_{i,j}=p_j/q_j$,  $j\equiv i\bmod 2$, 
	obtained from the respective two first finite 
	sums in the above representation of
	$x_0,x_1$. For example, for odd $j$ related to $x_1$ this reads
	\[
	\frac{p_j}{q_j}=\frac{p_{1,j}}{q_{1,j}}= \sum_{\ell=1}^{h_{j-1}} \frac{ f_{\ell}}{5^\ell} +
	\sum_{k=1}^{j-1}  \frac{a_k}{5^{h_k}}= 
	\frac{ \sum_{\ell=1}^{h_{j-1}} f_{\ell}5^{h_{j-1}-\ell} + \sum_{k=1}^{j-1} a_k5^{ h_{j-1}-h_k}  }{5^{h_{j-1}} }. 
	\]
	 Then essentially the same arguments from the proof of Theorem~\ref{th03} show that upon choosing $a_{j-1}$
	 appropriately (avoiding the residue class $-f_{h_{j-1}}$ modulo $5$ for each odd $j$, and $e_{h_{j-1}}$ for even $j$), the rationals $p_{j}/q_{j}$ are reduced with denominator $q_{j}=5^{h_{j-1}}$ and the error from the remaining infinite sums can be estimated 
	 \begin{equation}  \label{eq:tiere}
	 |x_1-\frac{p_{j}}{q_{j}}|= \left|\sum_{\ell>h_{j}} \frac{ f_\ell}{5^{\ell}} +
	 \sum_{k\geq j}  \frac{a_k}{5^{h_k}}\right| \ll 5^{-h_j},
	 \end{equation}
	 and similarly for $x_0$, i.e. 
	 as in \eqref{eq:FRi}. Hence $\mu(x_{i})\geq \lambda_{i}$, for $i=0,1$
	unless $x_{i}\in\mathbb{Q}$, can be derived as in the proof of Theorem~\ref{th03}. 
	Again the main difficulty is to show the converse
	$\mu(x_{i})\leq \lambda_{i}$, $i=0,1$. By symmetry, we restrict to $i=1$, and occasionally comment how to modify the arguments for $i=0$.
	
	Assume conversely, 
	for some $\mu>\lambda_{1}$ we have infinitely many $p/q$ with
	\begin{equation} \label{eq:bas}
	|x_{1}-\frac{p}{q}| \leq q^{-\mu}.
	\end{equation}
    Again we split into two cases according to the cases of
    rational approximations $p/q$ being among $p_{j}/q_{j}$ above for
    some odd $j$, or not.
    Case 1: $p/q=p_j/q_j$ for some odd $j$. Then 
    again $p_{j}/q_{j}$ are reduced and by 
    the same argument as in Theorem~\ref{th03} 
    satisfy $|x_1-p_j/q_j|\asymp 5^{-h_j}\asymp q_j^{-\lambda_1+o(1)}$ as $j\to\infty$ for a proper choice of $a_j$, 
    so \eqref{eq:bas} is impossible.
    
    Case 2: $p/q$ is not among $p_{j}/q_{j}$ above for odd $j$. 
    We will show that \eqref{eq:bas} implies $\mu(\xi)>2$, contradicting
    our hypothesis $\xi\in \mathscr{W}_{2}$.
    For given $p/q$, let again $u$ be the largest odd
    index with 
    $5^{h_{u-1}}=q_{u}<q$. Clearly $q_{u}<q\leq q_{u+2}$.
    Very similarly as in the proof of Theorem~\ref{th03}, Legendre Theorem implies
    \[
     q_{u}^{\lambda_{1}-1}\ll q\ll q_{u+2}^{1/(\mu-1)}<
     q_{u+2}^{1/(\lambda_{1}-1)}.
    \]
    Let $\Lambda=\lambda_{0}\lambda_{1}$.
    Now since $q_{u+2}\asymp q_{u}^{\Lambda}$, and as by \eqref{eq:FRi} 
    we have $q_u=5^{h_{u-1}}\asymp 5^{g_{u}}$,
    we derive
    \begin{equation} \label{eq:finaleole}
    5^{g_{u}(\lambda_{1}-1)} \ll q \ll  5^{\Lambda g_{u}/(\lambda_{1}-1)}.
    \end{equation}
    We remark that this implies $\lambda_0\geq (\lambda_1-1)^2/\lambda_1$,
    however we will not require this conclusion.
    
    The next important step is to construct good rational approximations to $\xi-x_1$.
    Since $u$ is odd, we may apply \eqref{eq:REP2} for $j=u$, and 
    when we split the third sum as
    \[
    \sum_{\ell>h_{j}} \frac{ e_\ell}{5^{\ell}}= 
    \sum_{\ell=h_{j}+1}^{h_{j+1}} \frac{ f_\ell}{5^{\ell}} +
    \sum_{\ell>h_{j+1}} \frac{ f_\ell}{5^{\ell}}
    \]
    and take the first summand of the last sum to the front, 
    we may rearrange
    \begin{align*}
    \xi-x_{1} &=\sum_{\ell\ge 1} \frac{d_{\ell}}{5^{\ell}} - 
    \left(\sum_{\ell=1}^{h_{u-1}} \frac{ f_{\ell}}{5^\ell} +
    \sum_{k=1}^{u-1}  \frac{a_k}{5^{h_k}} + \sum_{\ell>h_{u}} \frac{ f_\ell}{5^{\ell}} +
    \sum_{k\geq u}  \frac{a_k}{5^{h_k}}\right)\\
    &=\left( \sum_{\ell=1}^{h_{u}} \frac{d_{\ell}}{5^{\ell}} 
    +\sum_{\ell=h_{u}+1}^{h_{u+1}}  \frac{d_{\ell}}{5^{\ell}} 
    - 
    \sum_{\ell=1}^{h_{u-1}} \frac{ f_{\ell}}{5^\ell} 
    - \sum_{\ell=h_{u}+1}^{h_{u+1}} \frac{ f_{\ell}}{5^\ell}
    -
    \sum_{k=1}^{u}  \frac{a_k}{5^{h_k}}\right) \\ &+ 
    \sum_{\ell>h_{u+1}} \frac{d_{\ell}}{5^{\ell}}
    -\sum_{\ell>h_{u+1}} \frac{ f_\ell}{5^{\ell}} -
    \sum_{k\geq u+1}  \frac{a_k}{5^{h_k}}.
    \end{align*}
    But now by construction and since $u$ is odd, we have $d_{\ell}=f_{\ell}$ for $\ell\in I_{u+1}= [h_{u}+1,h_{u+1}]$.
    Hence the above expression simplifies to
    \[
    \xi-x_{1}= \left( \sum_{\ell=1}^{h_{u}} \frac{d_{\ell}}{5^{\ell}} 
    - 
    \sum_{\ell=1}^{h_{u-1}} \frac{ f_{\ell}}{5^\ell} 
    -
    \sum_{k=1}^{u}  \frac{a_k}{5^{h_k}}\right) + 
    \sum_{\ell>h_{u+1}} \frac{d_{\ell}}{5^{\ell}}
    -\sum_{\ell>h_{u+1}} \frac{ f_\ell}{5^{\ell}} -
    \sum_{k\geq u+1}  \frac{a_k}{5^{h_k}}.
    \]
    The finite sums in the brackets add up to a rational number $r/s$ 
    with denominator $s=5^{h_u}$ (we may again assume it is reduced
    upon avoiding some $a_j$, however we do not require this here).
    The remaining sums are obviously of order $\ll 5^{-h_{u+1}}\ll 5^{-\Lambda g_u}$,
    thus
    \[
    |\xi-x_{1}-\frac{r}{s}| \ll 5^{-\Lambda g_{u}}.
    \]
    Combining with \eqref{eq:bas} and $\mu>\lambda_{1}$ yields
    \begin{equation} \label{eq:eX}
    |\xi-(\frac{p}{q}+\frac{r}{s})| \leq |\xi-x_{1}-\frac{r}{s}|+|x_{1}-\frac{p}{q}|\ll \max\{ q^{-\mu}, 5^{-\Lambda g_{u}}  \} \ll \max\{ q^{-\lambda_{1}}, 5^{-\Lambda g_{u}}  \}.
    \end{equation}
    We distinguish two cases. Firstly assume the right bound in the maximum is larger, i.e. $q>5^{\lambda_{0}g_{u}}$.
    Then the rational number $M/N:=p/q+r/s=(ps+qr)/(qs)$ satisfies
    \[
    |\xi-\frac{M}{N}| \ll 5^{-\Lambda g_{u}}.
    \]
    On the other hand, as $u$ is odd, its
    denominator can be estimated as
    \begin{equation} \label{eq:RRr}
    N= qs=5^{h_{u}}q\ll 5^{\lambda_{1} g_{u}}q.
    \end{equation}
    Thus we have
    \[
    -\frac{\log |\xi-\frac{M}{N}|}{\log N} \geq 
    \frac{\log 5\cdot \Lambda g_{u}}{\log q+ \log 5\cdot \lambda_{1} g_{u}}.
    \]
    The right hand side is obviously decreasing in $q$ and thus in view 
    of the upper bound in \eqref{eq:finaleole} 
    after some calculation we derive
    \begin{equation}  \label{eq:modif}
    -\frac{\log |\xi-\frac{M}{N}|}{\log N} \geq \frac{\Lambda}{\frac{\Lambda}{\lambda_{1}-1}+\lambda_{1}}-\epsilon_{1}=
    \frac{\lambda_{0}(\lambda_1 -1)}{ \lambda_0+\lambda_1-1}-\epsilon_{1},
    \end{equation}
    with $\epsilon_{1}$ arbitrarily small for large enough $u$.
    The right hand side increases as a function of both
    $\lambda_0, \lambda_1$, hence a short calculation verifies
    that condition \eqref{eq:assuan} suffices for the conclusion
    $\mu(\xi)>2$ upon choosing $\epsilon_{1}$ small enough. Thus we have the desired contradiction to $\xi\in \mathscr{W}_{2}$.  
    (We notice that if
    only finitely many $M/N$ would occur, then $M/N=\xi$ is ultimately constant
    thus $\xi\in\mathbb{Q}$, again
    contradicting our hypothesis $\xi\in \mathscr{W}_{2}$.)
    A very similar argument applies for $i=0$, we end up
    at \eqref{eq:modif} with interchanged $\lambda_0, \lambda_1$, and
    the same argument applies.
    Finally assume the left expression in \eqref{eq:eX} is larger, thus
    $q\leq 5^{\lambda_{0}g_{u}}$.
    Then from \eqref{eq:eX}, \eqref{eq:RRr} we infer
    \[
    -\frac{\log |\xi-\frac{M}{N}|}{\log N} \geq 
    \frac{\lambda_{1} \log q }{\log q+ \log 5\cdot \lambda_{1}g_{u}}.
    \]
    Since the right hand side expression 
    increases in $q$, by \eqref{eq:finaleole} we conclude
    \begin{equation}  \label{eq:modif2}
    -\frac{\log |\xi-\frac{M}{N}|}{\log N} \geq \frac{\lambda_{1}\cdot (\lambda_{1}-1)}{(\lambda_{1}-1)+ \lambda_{1}}-\epsilon_{2}= \frac{\lambda_{1}^{2}-\lambda_{1}}{2\lambda_{1}-1}-\epsilon_{2}.
    \end{equation}
    Hence again $\mu(\xi)>2$ as soon as $\lambda_{1}>\rho$
    and $\epsilon_{2}$ is sufficiently small, again contradicting $\xi\in\mathscr{W}_{2}$. Similarly we require $\lambda_{0}>\rho$ for the contradiction
    when we start with $i=0$.
     Thus condition \eqref{eq:assuan} guarantees
    the implication in any case.
    
    Finally by a similar argument we show \eqref{eq:FF}.
    If we assume $\mu(\xi)>\tilde{\mu}$ for any given 
    $\tilde{\mu}\geq 2$ in place of $\mu(\xi)>2$, then similarly as above Case 1 cannot happen
    and in Case 2 we get a contradiction as soon as the right hand
    sides in \eqref{eq:modif}, \eqref{eq:modif2} exceed $\tilde{\mu}$, and
    similarly for $\lambda_0$ as well.
    This means upon these assumptions we have
    \[
    \mathscr{W}_{\lambda_0} + \mathscr{W}_{\lambda_1} \supseteq 
    \mathscr{W}_{2;\tilde{\mu}},
    \]
    or equivalently the complement of the sumset is contained in $\mathscr{W}_{\tilde{\mu};\infty}$. The left estimate in
    \eqref{eq:FF} thus follows
    from Jarn\'ik's formula \eqref{eq:jarnik} for the Hausdorff dimensions
    of the latter sets, and again using that the right
    hand sides in \eqref{eq:modif}, \eqref{eq:modif2} represent
    increasing functions. The upper bound $1$ in \eqref{eq:FF} 
    follows from the bound \eqref{eq:assuan} on the $\lambda_i$.
	\end{proof}

    \section{Proofs of generalisations in~\S~\ref{gener} }

    We first sketch how to extend
    Theorem~\ref{nonarxiv} to simultaneous approximation.

    \begin{proof}[Proof of Theorem~\ref{multidim}]
	Starting with a vector $\boldsymbol{\xi}\in\mathbb{R}^{m}$,
	we may apply the construction from the proof of Theorem~\ref{nonarxiv} coordinatewise with the same intervals $I_j$ for each coordinate
	to obtain $\boldsymbol{x}_0, \boldsymbol{x}_1$ that sum up to $\boldsymbol{\xi}$.
	We verify $\mu_m(\boldsymbol{x}_i)=\lambda_i$ for $i=0,1$.
	The lower bounds $\mu_{m}(\boldsymbol{x}_{i})\geq \lambda_{i}$
	follow analogously as for $m=1$, without restrictions
	on $\lambda_{i}$.
	For the reverse estimates, Case 1 follows again very similarly.
	In Case 2, we first find a generalisation
 of Proposition~\ref{kp} to general $m$, derived
from the one-dimensional case.

\begin{proposition}  \label{trotz}
   Let $\tau_k>1$. Let $1<q_k<q_{k+1}$ be integers and assume
     \[
   \max_{1\le i\le m}\vert \xi_i-p_{k,i}/q_k\vert<q_k^{-\tau_k}/2, \qquad \max_{1\le i\le m}\vert \xi_i-p_{k+1,i}/q_{k+1}\vert<q_{k+1}^{-1}/2
   \]
   hold
for rational vectors $\boldsymbol{p}_k/q_k, \boldsymbol{p}_{k+1}/q_{k+1}$ in reduced form. Then $q_{k+1}\gg q_{k}^{\tau_k-1}$. 
\end{proposition}

\begin{proof}
The deduction from Proposition~\ref{kp} 
is obvious if for some $i$ the pair $(q_k,p_{k,i})$ is coprime, but this need not be the case. Otherwise, fix some $i$ with largest common divisor 
$T= q_k^{a}>1$ for some $a\in (0,1]$ and consider the 
reduced fraction $p^{i,0}/q^0=(q_{k}/T)/(p_{k,i}/T)$
which is a convergent to $\xi_i$ since $\tau_k>1$.
We see that $q^0=q_k^{1-a}$ 
and so
\[
\vert p^{i,0}/q^0- \xi_i\vert= \vert p_{k,i}/q_k-\xi_i\vert< q_k^{-\tau_k} = (q^0)^{-\tau_k/(1-a) },
\]
thus by  Proposition~\ref{kp} the next convergent
to $\xi_i$ is of order 
\[
\gg (q^0)^{\tau_k/(1-a) - 1 }= q_k^{ (1-a)(\tau_k/(1-a) - 1) }= q_k^{\tau_k-1+a }
\ge q_k^{ \tau_k - 1 }.
\]
Hence by Legendre Theorem~\ref{lege} also $q_{k+1}\gg q_k^{ \tau_k - 1 }$.
\end{proof}

From Proposition~\ref{trotz}, we again conclude that \eqref{eq:finaleole} holds. Then we proceed precisely as in the one-dimensional setting. Since Lebesgue almost all $\boldsymbol{\xi}\in\mathbb{R}^{m}$
	are only simultaneously approximable of order $\mu_{m}(\boldsymbol{\xi})=1+\frac{1}{m}$, 
	according to \eqref{eq:modif}, \eqref{eq:modif2} we eventually 
	end up at the conditions \[
	\frac{\lambda_i(\lambda_{1-i}-1)}{\lambda_i+\lambda_{1-i}-1}>1+\frac{1}{m},\qquad  \frac{\lambda_{i}^{2}-\lambda_{i}}{2\lambda_{i}-1}>1+\frac{1}{m}
	\]
	for $i=0,1$.
	Thus, for the analogous full measure result, 
	as $m\to\infty$ the lower bound $\min \lambda_{i}>\gamma+o(1)$ suffices, giving rise to $\rho_m$ as in the theorem.
	Finally \eqref{eq:FF00}
	follows similarly from a well-known generalisation of \eqref{eq:jarnik} to higher dimension with right hand side $(m+1)/\lambda$, also due to Jarn\'ik~\cite{jarnik}.
\end{proof}

We explain the modifications to be made to obtain
the more general Theorem~\ref{ttt}.

\begin{proof}[Proof Theorem~\ref{ttt}]
	We define $x_0^{\ast}$ as in the proof of Theorem~\ref{nonarxiv} and again choose integers $a_k$ to analogously infer $x_0, x_1$, 
	however now do not restrict
	to $a_k\in\{0,1,2,3,4\}$ any longer. Assume for given $j$, we have already defined the initial elements $h_1, h_2,\ldots,h_{j-1}$ 
	and $a_1,\ldots, a_{j-1}$, and the corresponding reduced 
	$p_j/q_j$ with $q_j=5^{h_{j-1}}$.
	We define the next $h_j=-\lfloor \log \Phi_i(q_j)/\log 5\rfloor+j$ 
	 with $i\in\{ 0,1\}$ so that $i\equiv j\bmod 2$. Then
	 \begin{equation}  \label{eq:TRUCK}
	 \Phi_i(q_j)\asymp 5^{-h_j + j}.
	 \end{equation}
	If $i=1$, we split \eqref{eq:tiere} as
	 \begin{equation}  \label{eq:viecher}
	 |x_1-\frac{p_j}{q_j}|=| \sum_{\ell>h_j} \frac{f_{\ell}}{5^{\ell}} + \sum_{k\ge j}\frac{ a_j}{5^{h_j}}|=
	 | (\sum_{\ell>h_j} \frac{f_{\ell}}{5^{\ell}} + \frac{ a_j}{5^{h_j}})+
	 \sum_{k>j}  \frac{ a_k}{5^{h_k}}|
	 \end{equation}
	and accordingly for $i=0$. Let us for simplicity 
	assume $i=1$ from now on, the other case works analogously.
	We take the next $a_{j}$ so that the absolute value of
	the bracket expression
	\[
	E_{j}:=  \sum_{\ell>h_j} \frac{f_{\ell}}{5^{\ell}} + \frac{ a_j}{5^{h_j}}
	\]
	is 
	at most $\Phi_1(q_j)$ and as close as possible to it, 
	unless if the residue class condition modulo $5$ for $p_{j+1}/q_{j+1}$ 
	to be reduced in the next step (see proof of Theorem~\ref{nonarxiv}) fails, 
	in which case we alter it by $\pm 1$ to get the next nearest number
	of this form smaller than $\Phi_1(q_j)$. 
	It is clear that then 
	\begin{equation}  \label{eq:puha}
	0\le \Phi_1(q_j)-|E_j|\le 2\cdot 5^{-h_j}.
	\end{equation}
	First assume $a_j\ge 0$. Then clearly also $E_j\ge 0$ and
	\eqref{eq:puha} implies
	\[
	a_j/5^{h_j}\le \Phi_1(q_j) - \sum_{\ell>h_j} \frac{f_{\ell}}{5^{\ell}} \le \Phi_1(q_j) 
	\]
	hence by \eqref{eq:TRUCK} we infer
	\[
	0\le a_j\le 5^{h_j} \Phi_1(q_j) \ll 5^{j}.
	\]
	Thus, 
	the remaining expression in \eqref{eq:viecher} can be bounded by
	\begin{equation}  \label{eq:Zwerg}
	|\sum_{k>j} \frac{ a_k}{ 5^{h_k} }| \ll 5^{j-h_{j+1}}= o(5^{-h_j}),
	\end{equation}
where for the last estimate we used
	that the sequence $(h_j)_{j\ge 1}$ grows exponentially. Indeed,
	the decay assumption on $\Phi_1$ and \eqref{eq:TRUCK} lead to 
	\[
	5^{-h_j}\le 5^{-h_j+j}\ll \Phi_1(q_j)\le q_j^{-\lambda_i}\le 5^{ -h_{j-1}\rho }
	\]
	implying $h_j/h_{j-1} > \rho-o(1)>1$ for large $j$,
	and hence the claim.
	From \eqref{eq:viecher}, \eqref{eq:puha} and \eqref{eq:Zwerg},
	we infer that for large $j$ we may write
	\[
	|x_1-\frac{p_j}{q_j}| = \Phi_1(q_j)-R,
	\]
	for some remainder term $0\le R\le 3\cdot 5^{-h_j}$. 
	Thus, by choice of $h_j$,
	we infer
	\[
	0<1-\frac{|x_1-p_j/q_j|}{\Phi_1(q_j) } = \frac{R}{\Phi_1(q_j)}
	\ll \frac{5^{-h_j}}{\Phi_1(q_j)} \ll 5^{-j}.
	\]
	The expression tends to $0$ as $j\to\infty$. 
		Now assume $a_j<0$. Then $E_j<0$ and thus similarly 
	\eqref{eq:puha} implies the analogous estimate
	\[
	0\le |a_j|=-a_j\leq \Phi_1(q_j) - \sum_{\ell>h_j} \frac{f_{\ell}}{5^{\ell}}\le \Phi_1(q_j)
	\]
	and the same argument as above applies.
	A very 
	similar line of arguments applies to $i=0$. 
	Finally, in Case 2 of
	other rationals $p/q\neq p_j/q_j$, we get that $|x_i-p/q|/\Phi_i(q)\to \infty$
	from the same line of arguments as in Theorem~\ref{nonarxiv} upon
	our assumptions on the $\Phi_i$.
	Combining these facts yields $x_i\in Exact(\Phi_{i})$ for $i=0,1$.
	The last claim \eqref{eq:phan} on packing dimension again follows via Theorem~\ref{uppack} by 
	\[
	\dim_P(Exact(\Phi_1))\geq \dim( \prod_{i=0,1} Exact(\Phi_i)) - \dim(Exact(\Phi_0))\ge 1-\dim(Exact(\Phi_0))
	\] 
	and choosing any $\Phi_{0}$
	of very fast decay so that the last term vanishes by \eqref{eq:jarnik}.
	\end{proof}

We finally describe how to alter the construction of Theorem~\ref{nonarxiv} 
for missing digit Cantor sets.
The main obstacle is that we require $p_j/q_j$ to be reduced for Case 1. We 
may not be able to choose integers
$a_j$ as before without leaving the Cantor set for either $x_0$ or $x_1$. To avoid this problem, we slightly redefine our intervals $I_j$ and our restriction to $\mathscr{V}_{1,(b)}$ in \eqref{eq:friederich} enters. We keep the notation from the proof of Theorem~\ref{nonarxiv} above.

\begin{proof}[Proof of Theorem~\ref{PREIS}]
	Assume $0\in W$.
	Let $\xi\in \mathscr{V}_{1,(b)}\cap \mathscr{W}_2\cap K$ be arbitrary and $\varepsilon>0$. Write $\xi=\sum_{\ell\ge 1} d_{\ell}b^{-\ell}$ for its
	base $b$ representation, with $d_{\ell}\in W$. 
	The fact $\xi\in \mathscr{V}_{1,(b)}$ guarantees that there are no
	long blocks of consecutive $0$ digits
	in its base $b$ representation, more precisely
	in intervals $[N,(1+\varepsilon)N]$ for 
	any $\varepsilon>0$ and $N\geq N_0(\varepsilon)$. 
	Hence we can modify our sequences
	$g_j, h_j$ from the proof of Theorem~\ref{nonarxiv} so that the quotients $h_j/g_j$ 
	still tend to $\lambda_i$ with $i\in\{0,1\}$ so that $i\equiv j\bmod 2$, and	additionally at the first position $\ell=g_{j+1}=h_j+1$ 
	of $I_{j+1}$ the digit is in $d_{\ell}\in W\setminus \{0\}$. 
	Derive $x_0^{\ast}, x_1^{\ast}$ similar to Theorem~\ref{nonarxiv}
	via
		\[
	x_0^{\ast}= \sum_{\ell\ge 1} \frac{ e_{\ell}  }{b^{\ell}}, \qquad x_1^{\ast}= \sum_{\ell\ge 1} \frac{ f_{\ell}}{b^{\ell}},
	\] 
	where we let $e_{\ell}=d_{\ell}$ if 
	$\ell\in I_j$ for odd $j$ and $e_{\ell}=0$ otherwise, and vice versa
	for $f_{\ell}$.
	Now we simply let $x_0=x_0^{\ast}, x_1=x_1^{\ast}$, without twisting the digits at interval endpoints. Then for each even $j$
	we have
		\begin{equation*}  
	x_0= \sum_{\ell=1}^{h_{j-1}} \frac{ e_{\ell}}{b^\ell}  + \sum_{\ell> h_{j}} \frac{ e_\ell}{b^{\ell}}, \qquad e_{\ell}\in W,
	\end{equation*}
	and for odd $j$ we have
	\begin{equation*}  
	x_1= \sum_{\ell=1}^{h_{j-1}} \frac{ f_{\ell}}{b^\ell}+ \sum_{\ell> h_{j}} \frac{ f_\ell}{b^{\ell}}, \qquad f_{\ell}\in W.
	\end{equation*}
	Notice that the digits at positions $h_{j-1}+1,\ldots,h_j$ vanish.
	By construction, in the infinite sums for $j$ in question, at first 
	position $\ell=h_j+1$ the digit $e_{\ell}$ of $x_0$ resp. $f_{\ell}$
	of $x_1$ is non-zero.
	Notice that $x_0+x_1= x_0^{\ast}+x_1^{\ast}=\xi$ and $x_i\in K$.
	We define $p_j/q_j$ for $j\equiv i\bmod 2$ likewise as in Theorem~\ref{nonarxiv} 
	as the rational number obtained from the respective finite partial sums, in particular $q_j=b^{h_{j-1}}$.
	By the above properties, it is obvious
	that for $i=0,1$ and $j\equiv i \bmod 2$, as $j\to\infty$ we have
	\begin{equation} \label{eq:Quit}
	|x_i-p_j/q_j|\asymp b^{-h_{j}}\asymp b^{-\lambda_i h_{j-1} +o(1)} \asymp q_j^{-\lambda_i+o(1)},
	\end{equation}
	in particular $\mu(x_i)\ge \lambda_i$.
	For the reverse inequalities that settle \eqref{eq:friederich}, we again assume the opposite, namely we have $p/q$ that approximate some $x_i$
	of order larger than $\lambda_i$
	and consider the cases $p/q=p_j/q_j$ and $p/q$ not of
	this form separately. The latter Case 2 
	can be handled precisely as in the proof of Theorem~\ref{nonarxiv},
	here our assumption $\xi\in \mathscr{W}_2$ enters. 
	In Case 1, here we have the problem that we lack the twist
	with the $a_j$ to guarantee that $x_i$ have non-zero digit at positions
	$h_{j-1}$ when $j\equiv i\bmod 2$, and that consequently we cannot
	deduce that $p_j/q_j$ are reduced
	in the stated form with denominator $b^{h_{j-1}}$. 
	A priori it may happen that $p_j/q_j$
	is not reduced, and after reduction the approximation exponent
	may be larger than predicted. 
	
	To deal with this technical obstruction,
	we modify the argument as follows.
	Since $\theta_b(\xi)=1$,  given $\varepsilon>0$,
	for each large $j\ge j_0(\varepsilon)$
	at some slightly smaller position 
	\[
	z_{j-1}\in Z_{j-1}:= [h_{j-1}(1-\varepsilon),h_{j-1}]\cap \mathbb{Z}\subseteq I_{j-1},
	\]
	the number
	$\xi$ must have a non-zero digit, i.e. $d_{z_{j-1}}\in W\setminus \{0\}$. Assume $z_{j-1}$ is largest possible with this property.
	By construction, for $i\in\{0,1\}$ with $i\equiv j\bmod 2$, 
	the number $x_i$ has the same non-zero
	digit $d_{z_{j-1}}$ as $\xi$ at position $z_{j-1}$, followed by a digit string 
	of zeros up to (including) the last position of $I_{j}$.
	Write $r_j/s_j$ for the reduced fraction $p_j/q_j$.
	First
	assume $b$ is prime (more generally $(b,w)=1$ for all $w\in W\setminus\{0\}$ suffices). Then the above property 
	readily implies that
	the according fraction $p_j/q_j=p_j/b^{z_{j-1}}$ is almost reduced. 
	More precisely, after reduction it has denominator
	\begin{equation} \label{eq:gcd}
	s_j=b^{z_{j-1}}>b^{(1-\varepsilon)h_{j-1}}= q_j^{1-\varepsilon}.
	\end{equation} 
	Combining \eqref{eq:gcd} with \eqref{eq:Quit} readily 
	implies for $i=0,1$ and all large $j\equiv i\bmod 2$ that 
	\[
	|x_i-r_j/s_j|= |x_i-p_j/q_j|\gg q_j^{-\lambda_i+o(1)}
	\gg s_j^{(-\lambda_i+o(1))/(1-\varepsilon)} \gg s_j^{-\lambda_i-\epsilon_j},
	\]
	where
	$\epsilon_j>0$ tend to $0$ as $\varepsilon\to 0$ 
	and $j\to\infty$. 
	Since $\varepsilon$ can be taken arbitrarily small,
	we are done.
	If otherwise $b$ is composite, we possibly 
	have to slightly alter $x_0, x_1$ by interchanging their 
	base $b$ digits at certain places in the intervals $Z_{j-1}$ to
	guarantee $s_j>q_j^{1-\varepsilon}$ as in \eqref{eq:gcd} for the rationals analogously obtained 
	from the twisted numbers. 
	We claim that a suitable choice of digit positions
	is possible. This can be seen by first observing
	that iterating the argument for given $\varepsilon>0$, 
	we actually find many 
	positions in $Z_{j-1}$ where $\xi$ takes non-zero digits.
	Then considering variations of exchanging digits at some 
	of these places while keeping others, we finally see that some
	arising fractions must be almost reduced again. 
	We prefer to omit an exhaustive exposition of
	the slightly technical details. Moreover, the $x_i$ still sum up to $\xi$. Furthermore, the digit changes have only minor effect on the approximation quality $|x_i-r_j/s_j|$ since $|Z_{j-1}|$ is small compared to $|I_{j-1}|$.  Finally the 
	proof of Case 2 also remains essentially unaffected by these
	digital changes.
	
	 When $0\in W$, the implication \eqref{eq:KEIN} 
	follows	from the observation that the right hand side in \eqref{eq:friederich}
	still has full Hausdorff dimension $\dim(K)$. The latter
	claim is a consequence of the facts that almost all numbers in $C_{b,W}$ with respect to its natural Cantor measure (restricted Hausdorff measure of
	dimension $\log |W|/\log b$) satisfy both $\mu(\xi)=2$, proved by Weiss~\cite{weiss}, and $\theta_b(\xi)=1$, shown
	by Levesley, Salp and Velani~\cite[Corollary~1]{lsv}.
	The intersection clearly shares the same property. 
	The general case of \eqref{eq:KEIN} follows from the special case
	and Proposition~\ref{pp} via
	the rational shift map $x\to x-\min W/(b-1)$ that preserves
	the irrationality exponent $\mu$ and maps $C_{b,W}$ to 
	a missing digit Cantor set $C_{b,\tilde{W}}$ with $0\in \tilde{W}$.
	For the last implication \eqref{eq:WS}, very similar to
	Corollary~\ref{c3} we let $\lambda_0=\infty$, $\lambda_1=\lambda$
	in \eqref{eq:KEIN} and apply Theorem~\ref{uppack} and \eqref{eq:jarnik}.
	\end{proof}

   Due to the twist in the proof, we lose
   some flexibility regarding the order of approximation compared to 
   Theorem~\ref{nonarxiv}. In particular,
   we cannot guarantee an analogue of Theorem~\ref{ttt} for Cantor sets.
   From~\cite[Corollary~1]{lsv}, some weaker generalisation
   of Theorem~\ref{nonarxiv} can still
   be deduced, however we prefer not to state it here. 

   \vspace{1cm}

   {\em The author thanks the referee for the careful reading. I am further thankful to Sidney A. Morris for bringing to my attention his joint paper with Chalebgwa on Erdos-Liouville sets  }

\section{Declarations and data availability}
The author declares that no funds, grants, or other support were received during the preparation of this manuscript.\\

All data generated or analysed during this study are included in this published article

\end{document}